\documentclass{article} % Anonymized submission

\usepackage{amsmath,amssymb,amsthm}
\usepackage{bm}
\usepackage{comment}
\usepackage{microtype}
\usepackage{graphicx}
\usepackage{subfigure}
\usepackage{booktabs}
\usepackage{multirow}
\usepackage[page]{appendix}
\usepackage[flushleft]{threeparttable}
\usepackage{multirow}\usepackage{makecell}
\usepackage{enumerate}
\usepackage[section]{placeins}
\usepackage{tikz}

 \usepackage[accepted]{icml2023}
\usepackage{bookmark}

\icmltitlerunning{Decentralized Riemannian natural gradient methods with Kronecker-product approximations}

\newtheorem{thm}{Theorem}%[section] %(If you want theorem numbered
\newtheorem{lem}{Lemma}%[section] %%    with section number.
\newtheorem{defi}{Definition}%[section]

\newtheorem{assum}{Assumption}

\newtheorem{remark}{Remark}

\newcommand{\R}{\mathbb{R}}

\newcommand{\be}{\begin{equation}}
	\newcommand{\ee}{\end{equation}}
\newcommand{\en}{\begin{equation*}}
	\newcommand{\een}{\end{equation*}}
\newcommand{\eqn}{\begin{eqnarray}}
	\newcommand{\eeqn}{\end{eqnarray}}
\newcommand{\bmat}{\begin{bmatrix}}
	\newcommand{\emat}{\end{bmatrix}}
\newcommand{\btab}{\begin{tabular}}
	\newcommand{\etab}{\end{tabular}}

\newcommand{\iprod}[2]{\left \langle #1, #2 \right \rangle }

\newcommand{\E}{\mathbb{E}}

\DeclareMathOperator*{\argmin}{\text{argmin}}

\newcommand{\grad}{\operatorname{grad}}

\newcommand{\calM}{\mathcal{M}}

\newlength{\imgwidth}
\newcommand{\revise}[1]{#1}

\begin{document}
\onecolumn
\icmltitle{Decentralized Riemannian natural gradient methods with Kronecker-product approximations}

\icmlsetsymbol{equal}{*}

\begin{icmlauthorlist}
    \icmlauthor{Jiang Hu}{goo}
	\icmlauthor{Kangkang Deng}{to}
    \icmlauthor{Na Li}{to1}
	\icmlauthor{Quanzheng Li}{goo}
\end{icmlauthorlist}

% \author{Jiang Hu \thanks{Massachusetts General Hospital and Harvard Medical School, Harvard University, US. Email addresses: {\tt\small hujiangopt@gmail.com}}   
%   \and 
%   Kangkang Deng \thanks{Beijing International Center for Mathematical Research, Peking University, China. Email addresses: {\tt\small freedeng1208@gmail.com}}
%   \and
%   Na Li \thanks{School of Engineering and Applied Science, Harvard University, US. Email addresses: {\tt\small nali@seas.harvard.edu}}
%   \and
%   Quanzheng Li \thanks{Massachusetts General Hospital and Harvard Medical School, Harvard University, US. Email addresses: {\tt\small li.quanzheng@mgh.harvard.edu}}
% }
\icmlaffiliation{to}{Beijing International Center for Mathematical Research, Peking University, China (freedeng1208@gmail.com)} 
\icmlaffiliation{to1}{School of Engineering and Applied Science, Harvard University, USA (nali@seas.harvard.edu)} 
\icmlaffiliation{goo}{Massachusetts General Hospital and Harvard Medical School, Harvard University, USA (hujiangopt@gmail.com, li.quanzheng@mgh.harvard.edu)}

\icmlcorrespondingauthor{Kangkang Deng}{freedeng1208@gmail.com}

\icmlkeywords{Machine Learning, ICML}

\vskip 0.3in



\printAffiliationsAndNotice{}  % leave blank if no need to mention equal contribution
\begin{abstract}
With a computationally efficient approximation of the second-order information, natural gradient methods have been successful in solving large-scale structured optimization problems. We study the natural gradient methods for the large-scale decentralized optimization problems on Riemannian manifolds, where the local objective function defined by the local dataset is of a log-probability type. By utilizing the structure of the Riemannian Fisher information matrix (RFIM), we present an efficient decentralized Riemannian natural gradient descent (DRNGD) method. To overcome the communication issue of the high-dimension RFIM, we consider a class of structured problems for which the RFIM can be approximated by a Kronecker product of two low-dimension matrices.  By performing the communications over the Kronecker factors, a high-quality approximation of the RFIM can be obtained in a low cost. We prove that DRNGD converges to a stationary point with the best-known rate of $\mathcal{O}(1/K)$. 
Numerical experiments demonstrate the efficiency of our proposed method compared with the state-of-the-art ones. To the best of our knowledge, this is the first Riemannian second-order method for solving decentralized manifold optimization problems. 
\end{abstract}

\section{Introduction}\label{sec:intro}
We focus on a class of optimization problems on Riemannian manifold with negative log-probability losses, namely,
\be \label{prob} \min_{\Theta \in \calM} \;\; \varphi(\Theta):=\frac{1}{N} \sum_{i=1}^N \left\{ \frac{1}{|\mathcal{D}_i|}\sum_{j\in \mathcal{D}_i} \underbrace{-\log p(y_j| f(x_j, \Theta))}_{=:\varphi_j(\Theta)} \right\}, \ee
where $\calM$ is an embedded submanifold of $\R^{n\times d}$ 
or a quotient manifold whose total space is an embedded submanifold of $\R^{n\times d}$, $N$ is the number of total agents and each agent $i$ holds the local dataset $\{(x_j,y_j)\}_{j \in \mathcal{D}_i}$ with cardinality $|\mathcal{D}_i|$, $f$ is a mapping from the input space to the output space, and $p(\cdot| f(x_j,\Theta))$ is the probability density function conditioning on $f(x_j, \Theta)$. For the decentralized (multi-agent) setting, problem \eqref{prob} is equivalently rewritten as the following
\be \label{prob:decen}
\begin{aligned}
\min \;\; & \frac{1}{N} \sum_{i=1}^N \left\{ \frac{1}{|\mathcal{D}_i|}\sum_{j\in \mathcal{D}_i} -\log p(y_j| f(x_j, \Theta_i)) \right\}, \\
\mbox{s.t.} \;\; & \Theta_1 =\Theta_2 = \cdots = \Theta_N,
\\ 
& \Theta_i \in \calM, \forall i=1,\ldots,N,
\end{aligned}
 \ee
where $\Theta_i$ denotes the local copy of the $i$-th agent.

When the conditional distribution is Gaussian, the function $\log p$ is the mean-square loss. If the conditional distribution is the multinomial distribution, the function $\log p$ is the cross-entropy loss. The equivalence between the negative log-probability loss and Kullback-Leibler divergence, and more connections between general distributions and loss functions can be found in \citep{martens2020}. Therefore, problem \eqref{prob} encompasses a wide range of optimization problems arising from machine learning, signal processing, and deep learning, see, e.g., the low-rank matrix completion \citep{boumal2015low,kasai2019riemannian}, the low-dimension subspace learning \citep{ando2005framework,mishra2019riemannian}, and the deep neural networks with batch/layer normalization \citep{cho2017riemannian,hu2022riemannian,ba2016layer}, where the network parameters lie on the product of the Grassmann manifolds and the Euclidean spaces.

\subsection{Literature review}
Decentralized optimization in the Euclidean space (i.e., $\calM = \R^{n\times d}$) has been extensively studied in the last decades. Decentralized (sub)-gradient descent (DGD) method \citep{tsitsiklis1986distributed,nedic2009distributed,yuan2016convergence} gives the simplest way to combine local gradient update and consensus update. The convergence therein shows DGD can not converge to a stationary point when fixed step sizes are used. To obtain exact convergence with fixed step sizes, the local historic information are investigated in various papers, such as the primal-dual methods, EXTRA \citep{shi2015extra}, DLM \citep{ling2015dlm}, exact diffusion \citep{yuan2018exact}, NIDS \citep{li2019decentralized}, and the gradient tracking-type methods \cite{xu2015augmented,qu2017harnessing,scutari2019distributed}. The convergence rate can be linear if the objective function is strongly convex. To accelerate the convergence, many decentralized second-order methods are developed based on the local Hessian information or its quasi-Newton approximations. Although the second-order methods \citep{mokhtari2016network,eisen2017decentralized,mokhtari2016dqm,zhang2021newton,li2020communication} show better numerical performances than the first-order ones, only linear convergence rate is proved for the strongly convex setting. The main bottleneck is the use of the linear consensus algorithm. By a finite-time consensus method, an adaptive Newton method is proposed in \citep{zhang2022distributed}. The global superlinear convergence is established by assuming strong convexity. Decentralized stochastic second-order methods with variance reduction are also investigated in \citep{zhang2022variance}.

Due to the nonlinearity \revise{and nonconvexity} of the manifold constraint, the usual Euclidean consensus step can not achieve consensus. \revise{The above decentralized second-order methods rely on the convexity of either the considered problem or its feasible domain, and then can not be applied due to the failure of achieving consensus.} By utilizing the geodesic distance, the Riemannian consensus is defined in \citep{tron2012riemannian}. However, the costly exponential map and vector transport are involved and an asymptotically infinite number of consensus steps is needed to obtain the convergence \citep{shah2017distributed}. Based on a penalized reformulation, a Riemannian gossip algorithm on the Grassmann manifold is proposed in \cite{mishra2019riemannian}. For the specific Stiefel manifold (i.e., $\calM = {\rm St}(n,d)$), a new Riemannian consensus is defined based on the Euclidean distance in \citep{chen2021local}. \revise{They show that  the consensus problem on the Stiefel manifold yields some kind of generalized convexity and the Riemannian gradient method converges linearly to achieve the consensus  on the manifold.} By applying a multi-step consensus algorithm, the decentralized Riemannian stochastic gradient method and decentralized Riemannian gradient tracking method are proposed in \citep{chen2021decentralized} with the convergence rate of $\mathcal{O}(1/K)$ to a stationary point. By defining an approximate augmented Lagrangian function, the authors \citep{wang2022decentralized} propose a decentralized gradient tracking method, which does not need multi-step consensus. 

\subsection{Our Contributions}\label{subsec:contri}
We present an efficient decentralized Riemannian natural gradient method for solving \eqref{prob}. The main contributions are summarized as follows.
\begin{itemize}
    \item We investigate the Kronecker-product approximation of the RFIM by the factorized form of the gradient. 
    % By requiring the gradient of $\varphi_i(\Theta)$ is approximately low-rank, an approximate RFIM is constructed by the Kronecker products of two low-dimension matrices. 
    Instead of directly communicating the RFIM, we apply the gradient-tracking type technique on the two Kronecker factors. The communication costs can be as cheap as the gradient tracking procedure due to the low dimensionality of the Kronecker factors. By utilizing a switching rule between the communicated RFIM and local RFIM, we can always guarantee the descent property of the local natural gradient direction. Then, a decentralized Riemannian natural gradient descent method is proposed. We summarize our approach in Figure \ref{fig:illu}. 
    
    \begin{figure}[!htp] 
    \centering
    \begin{tikzpicture}[scale = 1] 
    % \node (rect) at (0,0) [draw, thick, fill=orange, fill opacity=0.7,minimum width=1cm,minimum height=4cm] {$Y$};
    % \node at (1,0) {$=$};
    % \draw[thick] (1.5,-2) rectangle (4.5,1); 
    % \fill[blue, opacity=0.3] (1.5,-0.5) rectangle (3.5,2);
    % \draw[thick]
    % (2,2) -- (2,2.6) -- (6,2.6) -- (6,2) -- cycle;
    % \node at (4,2.3) {\scriptsize{The gradient is approximately low-rank.}};
    % \node at (4, 1.5) {$\Downarrow$};
    \draw[thick, fill=blue, fill opacity=0.3] 
    (1,-1) -- (1,1) -- (3,1) -- (3,-1) -- cycle; 
    \node at (2,0) {\tiny{$F \in \R^{(nd) \times (nd)}$}};
    \node at (3.3,0) {$\approx$};
    \draw[thick, fill=green, fill opacity=0.3] 
    (3.6,-0.6) -- (3.6,0.6) -- (4.8,0.6) -- (4.8,-0.6) -- cycle; 
    % \tikzstyle{greenrect} = [rectangle, draw, fill=green, fit=1];
    % \coordinate (A) at (4,1);
    % \coordinate (B) at (5,-1);
    % \node[greenrect={(A) (B)}] {}; 
    \node at (4.2,0) {\tiny{$A \in \R^{d \times d}$}};
    \node at (5.05,0) {$\otimes$};
    \draw[thick, fill=orange, fill opacity=0.3] 
    (5.3,-0.7) -- (5.3,0.7) -- (6.7,0.7) -- (6.7,-0.7) -- cycle;
    \node at (6,0) {\tiny{$B \in \R^{n\times n}$}};
    \node at (4,-1.2) {$\Downarrow$};
    \draw[thick] (4,-2.2) circle (0.4);
    \node at (4, -2.2) {\tiny{${\color{green} A_1}, {\color{orange} B_1}$}};
    \draw[thick] (2,-3.2) circle (0.5);
    \node at (2, -3.2) {\tiny{${\color{green} A_N}, {\color{orange} B_N}$}};
    \draw[thick] (6,-3.2) circle (0.4);
    \node at (6, -3.2) {\tiny{${\color{green} A_2}, {\color{orange} B_2}$}};
    \draw[thick] (5,-4.2) circle (0.4);
    \node at (5, -4.2) {\tiny{${\color{green} A_3}, {\color{orange} B_3}$}};
    \node at (3,-4.2) {$\ldots$};
    \draw[<->] (4.4,-2.4) -- (5.6,-3);
    \draw[<->] (2.5,-3.2) -- (5.6,-3.2);
    \draw[<->] (4.2,-2.6) -- (4.8,-3.8);
    \draw[<->] (2.4,-3.5) -- (4.6,-4.2);
    \draw[<->] (3.4,-4.2) -- (4.6,-4.2);
    \draw[<->] (3.3,-4.1) -- (5.6,-3.2);
    \end{tikzpicture}
    \caption{Illustrations of Kronecker-product approximation and decentralized Riemannian natural gradient method.}
    \label{fig:illu}
    \end{figure}
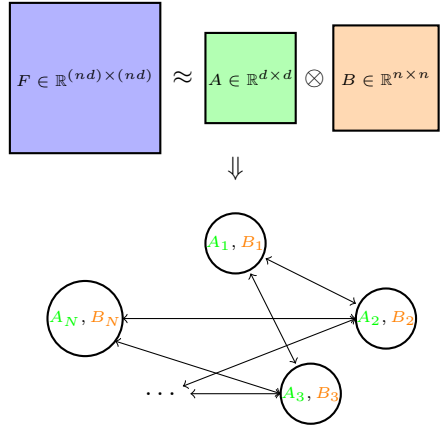
    
    \item The convergence of the proposed algorithm with rate of $\mathcal{O}(1/K)$ to a stationary point (Theorem \ref{thm:convergence}) is established for the cases of $\calM$ being the Stiefel manifold or the Grassmann manifold.  Compared to the analysis of the Riemannian gradient tracking method \citep{chen2021decentralized}, our result is applicable to any preconditioned direction with bounded curvatures and does not limit to the case of the Stiefel manifold. In particular, we show in Lemma \ref{lem:upper-bound-grass} that the Riemannian upper bound condition and the Lipschitz continuity of the Riemannian gradient hold for the Grassmann manifold, which is vital in analyzing the convergence. 
    % From the observation that the quadratic upper bound and the Lipschitz continuity of the Riemannian gradient of the Grassmann manifold, we show  \revise{Add the differences, Upper bound on Grassmann manifold, Lipschitz continuity, generalization to any descent direction.} The iteration complexity matches the result of the Riemannian gradient tracking method in \citep{chen2021decentralized}.
    \item Taking the low-rank matrix completion and the low-dimension subspace learning as examples, we detail the involved computational costs with a comparison with the gradient tracking. Numerical experiments (see Figure \ref{fig:sl:sarcos1} for an illustration) are performed to show the efficiency of the proposed algorithm compared with the Riemannian gradient tracking method.
\end{itemize}
\begin{figure}[!htb]
	\begin{center}
		\begin{minipage}[b]{0.48\linewidth}
			\centering
			\centerline{\includegraphics[width=\linewidth]{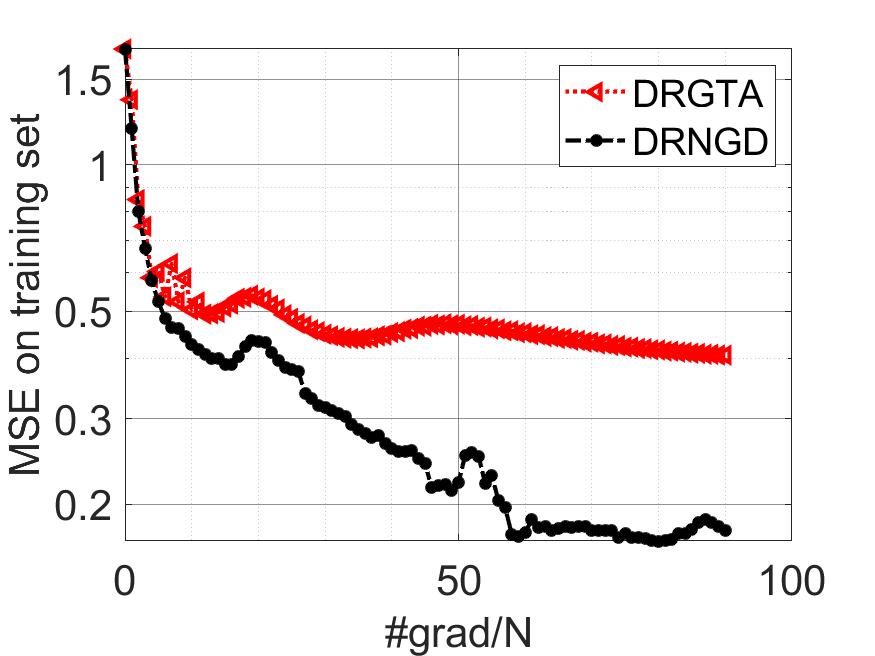}}
			%  \vspace{1.0cm}
			\centerline{}\medskip
		\end{minipage}
		%\hfill
		\begin{minipage}[b]{0.48\linewidth}
			\centering
			\centerline{\includegraphics[width=\linewidth]{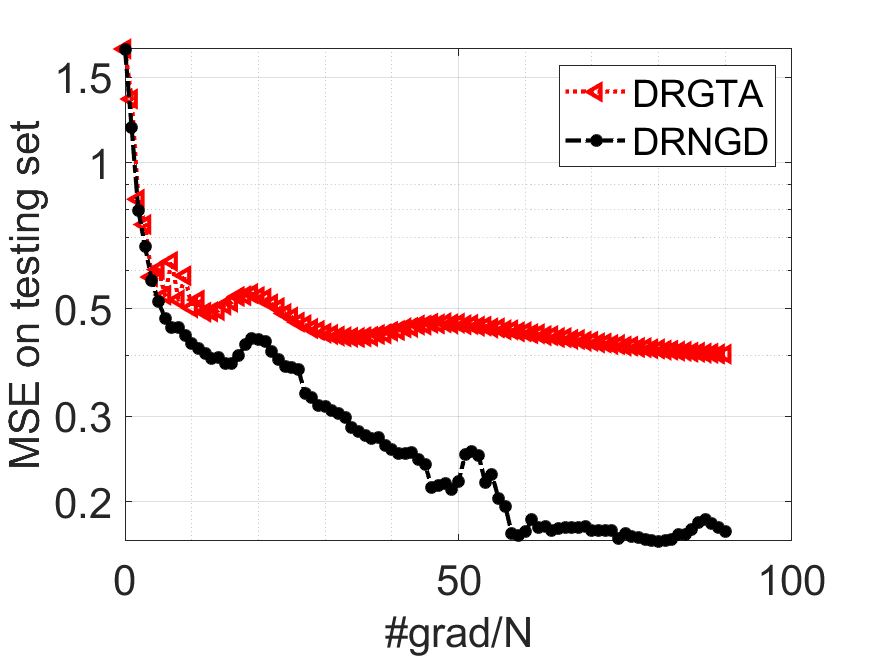}}
			%\cspace{1.5cm}
			\centerline{} \medskip
		\end{minipage}
	\end{center}
	\vskip -0.4in
	\caption{Numerical results on the low-dimension subspace learning problem with Sarcos dataset. ``DRNGD'' and ``DRGTA'' denote our decentralized Riemannian natural gradient method and the decentralized Riemannian gradient tracking method in \citep{chen2021decentralized}, respectively.}
% 	\lina{need to add grid for all figures.}
	\label{fig:sl:sarcos1}
\end{figure}
\subsection{Notation}
For a positive integer $T$, we denote $[T]=\{1,2,\ldots, T\}$. For a matrix $\Theta \in \R^{n\times d}$, $\|\Theta \|$ is denoted as the Frobenius norm of a matrix $\Theta$. \revise{For two matrices $\Theta_1,\Theta_2\in \mathbb{R}^{n\times d}$, we denote the Euclidean product $\iprod{\Theta_1}{\Theta_2}: = \text{tr}(\Theta_1^{\top}\Theta_2)$, where $\text{tr}(A)$ is the sum of the diagonal elements of $A$.} For a real-valued function $h$, we use $\nabla h(\Theta)$ and $\grad h(\Theta)$ to denote the Euclidean and Riemannian gradient under the Euclidean metric, respectively. We use $\otimes$ to denote the Kronecker product. We write $\theta = {\rm vec}(\Theta) \in \R^{nd}$ as the column-wise vectorization of matrix $\Theta$. For simplicity, the vectorizations of $\nabla h(\Theta)$ and $\grad h(\Theta)$ are denoted by $\nabla h(\theta)$ and $\grad h(\theta)$, respectively. 

\section{Preliminaries}\label{sec:preli}
\subsection{Manifold optimization}
Manifold optimization is concerned with the following problem:
\[ \min_{\Theta \in \calM} \quad h(\Theta), \]
where $\calM$ is a Riemannian manifold with its total space included in $\R^{n\times d}$  and $h$ is a real-valued function. It has been vastly studied over the last decades, see, e.g., \citep{absil2009optimization,hu2020brief,boumal2020introduction}. A core concept in designing the Riemannian optimization algorithms is the retraction operator. Let $T_{\Theta} \calM$ be the tangent space at $\Theta$ to $\calM$, the retraction $R_{\Theta}$ is a mapping from $T_{\Theta} \calM$ to $\calM$ and satisfies the following two conditions:
\begin{itemize}
    \item $R_{\Theta}(0_{\Theta}) = \Theta$, where $0_{\Theta}$ is the zero element in $T_{\Theta} \calM$.
    \item $\frac{\rm d}{\rm dt} R_{\Theta}(t Q)|_{t=0} = Q$ for any $Q \in T_{\Theta}\calM$.
\end{itemize}
Taking the Stiefel manifold ${\rm St}(n,d) := \{\Theta \in \R^{n \times d}: \Theta^\top \Theta = I\}$ as an example, the polar decomposition $R_{\Theta}(Q):= UV^\top$ with $U,V$ being the left and right singular matrices of the matrix $\Theta + Q$ (i.e., there exists a diagonal matrix $S$ with positive entries on its diagonal such that $\Theta + Q = USV^\top$) is a retraction.  

In the $k$-th iteration, the update scheme of the Riemannian algorithm obeys the following form
\[ \Theta_{k+1} = R_{\Theta_k}(\alpha_k \eta_k),  \]
where $\eta_k$ is a descent direction in $T_{\Theta_k} \calM$ and $\alpha_k > 0$ is a step size. When the vectorization of $\Theta$, i.e., $\theta \in \R^{nd}$, is used, the notations $q \in T_{\theta} \calM$ and $R_{\theta}(q)$ mean that 
% \lina{should clarify the dimension of $\theta$, $q$. Is $\theta\in\mathbb{R}^{nd\times 1}$?}
\[ {\rm mat}(q) \in T_{\Theta} \calM, \quad R_{\theta}(q) = {\rm vec}(R_{\Theta}({\rm mat}(q))),  \]
where ${\rm mat}(q)$ converts $q$ into a $n$-by-$d$ matrix and satisfies ${\rm vec}({\rm mat}(q)) = q$. Two specific manfiolds, the Stiefel manifold ${\rm St}(n,d) := \{\Theta \in \R^{n\times d}: \Theta^\top \Theta = I \}$ and the Grassmann manifold ${\rm Gr}(n,d) = \{ {\rm all~}d{\rm-dimension ~ subspaces ~in~} \R^n\}$, will be considered in this paper. More explanations on these two manifolds are presented in Appendix \ref{append:prem}.
% \lina{should explain the operator ``$/$" in the definition of Grassmann manifold.} \lina{In supplementary it would be better to  add some materials for the two manifolds.} 

\subsection{Riemannian Fisher information matrix and natural gradient methods}
The Fisher information matrix (FIM) and natural gradient method are proposed in \citep{amari1996neural}. Lately, the theoretical properties and numerical performances of the natural gradient methods and their generalizations are extensively studied; see, e.g., \citep{martens2015optimizing,martens2020,yang2022efficient}. When it comes to the manifold case, the authors \citep{hu2022riemannian} define the RFIM (i.e., Riemannian FIM) for \eqref{prob} 
% \lina{what's R in RFIM? } 
by 
\be \label{eq:fim-org} \begin{aligned}
F(\theta):= & \frac{1}{N}\sum_{i=1}^N \left\{\frac{1}{|\mathcal{D}_i|} \sum_{j\in \mathcal{D}_i}
\mathbb{E}_{p(y|f(x_j,\theta))} \right. \\
& \left.\left[ \grad \log p(y|f(x_j,\theta)) \grad \log p(y|f(x_j,\theta))^\top \right] \right\},
\end{aligned}  \ee
which is a $(nd)$-by-$(nd)$ matrix. When $\calM$ is a submanifold of $\R^{n\times d}$, $F(\theta)$ coincides with the Riemannian Hessian by \citep{hu2022riemannian}. 
When the expectation $\mathbb{E}_{p(y|f(x_j,\theta))}$ is expensive to compute, the empirical RFIM \citep{hu2022riemannian} is defined by 
\[ \bar{F}(\theta):= \frac{1}{N}\sum_{i=1}^N \left[ \frac{1}{|\mathcal{D}_i|} \sum_{j\in \mathcal{D}_i} \grad \varphi_j(\theta) \grad \varphi_j(\theta)^\top \right]. \]
One can refer to \citep{martens2020,hu2022riemannian} for the comparisons of the empirical RFIM and the RFIM.  Both the empirical RFIM and the RFIM play important roles in the design of natural gradient methods \citep{martens2015optimizing,martens2020,hu2022riemannian}.

% \lina{explain how good is the empirical RFIM? who proposed this empirical RFIM? cite the reference.}

\subsection{Stationarity measure}
Let $\Theta_1, \ldots, \Theta_N \in \calM$ be the local copies of each agent, and denote the Euclidean mean by 
\[ \hat{\Theta} := \frac{1}{N} \sum_{i=1}^N \Theta_i. \]
The induced arithmetic mean on manifold \citep{sarlette2009consensus} is defined by
\[ \bar{\Theta} := \argmin_{y \in \calM} \;\; \sum_{i=1}^N\|y-\Theta_i\|^2 = P_{\calM}(\hat{\Theta}), \]
where $P_{\calM}$ is the orthogonal projection operator onto $\calM$. Then, the $\varepsilon$-stationary point is defined as follows.
\begin{defi}
The set of points $\{\Theta_1, \ldots, \Theta_N\} \subset \calM$ is called a $\varepsilon$-stationary point of \eqref{prob:decen} if 
\[ \frac{1}{N}\sum_{i=1}^N \| \Theta_i - \bar{\Theta}\|^2 \leq \varepsilon \quad {\rm and} \quad \|\grad \varphi(\bar{\Theta})\|^2 \leq \varepsilon. \] 
\end{defi}

\section{Decentralized Riemannian natural gradient method}
% \lina{This section needs major revision. it will be difficult for readers to know the algorithms. We should have more subsections. 1) Decentralized setup (see my next comment); 2) Algorithm design rational (explain these A, B, gradient tracking); 3) Our Algorithm (besides the concrete alg), add a few sentences explaining the algorithms. 4) Two practical problems, your current section of ``low-rank matrix completion" and "low-dimension subspace". 5) Convergence analysis.} \\

% \lina{As I said earlier, we should formally introduce the centralized Riemannian natural gradient methods. Then we need to formally introduce our requirement for ``decentralized'' methods. Then discuss the challenges in making the centralized methods to be decentralized. When introducing the settings/requirement for ``decentralized'', need to introduce the communication network and communication matrices.}

\subsection{Centralized Riemannian natural gradient descent method}
Assuming that $F(\theta)$ is positive definite, the Riemannian natural gradient direction $d(\theta)$ is defined as 
\be \label{eq:ng} d(\theta) := F(\theta)^{-1}\grad \varphi(\theta). \ee
% \lina{We should use the display model for the equation to formally introduce Riemannian natural gradient methods. We should either do it here or the beginning of section 2.4. }
Then, at the $k$-th iteration, the centralized Riemannian natural gradient descent method has the following update scheme,
\be \label{eq:ngd} \theta_{k+1} = R_{\theta_k}(-\beta d_k), \ee
with $d_k=d(\theta_k)$ and a step size $\beta > 0$. 
The success of natural gradient methods based on \eqref{eq:ng} owes to the cheap computations as well as the connections to the Hessian matrix.

Let us start with some basics of decentralized optimization before introducing the decentralized extension of the Riemannian natural gradient descent method. In the literature of decentralized optimization, the network topology of the agents is modeled by its adjacency matrix $W \in \R^{N\times N}$, which is nonnegative and symmetric. The agents $i$ and $j$ are connected if $w_{ij} > 0$ and otherwise not. In addition, $W$ is also assumed to be doubly stochastic, i.e., for any $i\in [N]$, $\sum_{j=1}^N w_{ij} =1$. The $i$-th agent can communicate with its neighboring agents in the network, i.e., any node $j$ with $w_{ij} > 0$. In the $k$-th iteration, let $\theta_{i,k}$ be the local iterate of the $i$-th agent. 
With the adjacency matrix $W$,
% We are now introducing the decentralized Riemannian natural gradient method for solving \eqref{prob}. Let us start with the gradient tracking method. 
the gradient tracking method \citep{chen2021decentralized} constructs the following estimated gradient for the $i$-th agent
\be \label{eq:gt} y_{i,k} = \sum_{j=1}^N w^t_{ij} y_{i,k-1} + g_{i,k} - g_{i,k-1}, \ee
where $w^t_{ij}$ is the $(i,j)$-th element of $W^t$, $g_{i,k}:= \frac{1}{|\mathcal{D}_i|} \sum_{j \in \mathcal{D}_i} \grad \varphi_j(\theta_{i,k})$, $y_{i,k-1}$ denotes the estimated gradient used in the $(k-1)$-th iteration at the $i$-th agent, and $y_{i,0}$ is set to $g_{i,0}$. Then, the Riemannian gradient tracking method is given by 
\[ \theta_{i,k + 1} = R_{\theta_{i,k}}\left( \alpha P_{T_{\theta_{i,k}}\calM}\left(\sum_{j=1}^N w^t_{ij}\theta_{j,k} \right) -  \beta v_{i,k} \right),  \]
where $v_{i,k} = P_{T_{\theta_{i,k}}\calM}(y_{i,k})$, $\alpha > 0$ and $\beta > 0$ are the step sizes. 

Note that the centralized Riemannian natural gradient method \eqref{eq:ngd} relies on all local information to define $F(\theta)$. In the gradient tracking \eqref{eq:gt}, the communication matrix is a $nd$-dimension vector $y_{i,k}$. Due to the high dimensionality of $F(\theta) \in \R^{(nd) \times (nd)}$, it will be extremely costly to get an accurate approximation of $F(\theta)$ in the local agent if directly communicating $F(\theta)$ by a gradient tacking-type technique. We will present an approach to do efficient communication based on the Kronecker-product approximation of the RFIM.

% \lina{so far, the algorithm is gradient tracking; unclear how RFIM plays in a role. You need to have a centralized natural gradient method earlier so people know the role of RFIM and what is needed for each agent for RFIM to perform natural gradients. }

\subsection{Communications of the RFIM based on the Kronecker-product approximation}
As defined in \eqref{eq:fim-org}, the RFIM $F(\theta)$ is an $(nd)$-by-$(nd)$ matrix. Due to its high dimensionality, direct communication on $F(\theta)$ is generally not affordable. Here, we consider a class of problems of form \eqref{prob} whose RFIM can be approximated by the Kronecker product of two smaller-scale matrices. Define $\varphi(\Theta;x,y) := -\log p(y|f(x,\Theta))$. As shown in \citep{martens2015optimizing,grosse2016kronecker,hu2022riemannian}, if the gradient $\nabla \varphi(\Theta;x,y)$ with  has the approximate decomposition
\be \label{eq:lowrank-grad} \nabla \varphi(\Theta; x, y) \approx B(\Theta; x,y) A(\Theta; x,y)^\top, \ee
with a positive integer $r$, $B(\Theta;x,y) \in \R^{n\times r}$, and $A(\Theta; x,y) \in \R^{d \times r}$, an approximate RFIM can be constructed by
\be \label{eq:RFIM-kron} F(\theta) \approx  A^F(\theta) \otimes [Q(\Theta) B^F(\theta) Q(\Theta)] \ee
with $A^F(\theta):= \frac{1}{N}\sum_{i=1}^N \left\{ \frac{1}{|\mathcal{D}_i|} \sum_{j \in \mathcal{D}_i} \E_{p_{y|f(x_j,\theta)}} [A(\Theta;x_j,y) \right. $ $ \left. A(\Theta;x_j,y)^\top] \right\}$ $\in$ $\R^{d \times d}$, $B^F(\theta):= \frac{1}{N}\sum_{i=1}^N \left\{ \frac{1}{|\mathcal{D}_i|} \right.  $ $ \left.\sum_{j \in \mathcal{D}_i} \E_{p_{y|f(x_j,\theta)}} [B(\Theta;x_j,y) B(\Theta;x_j,y)^\top] \right\}  \in \R^{n\times n}$, and $Q(\Theta) \in \R^{n\times n}$ is matrix representation of the projection operator $P_{T_{\Theta} \calM}$, i.e., $P_{T_{\Theta} \calM}(U) = Q(\Theta)U$ (e.g., $Q(\Theta) = I - \Theta \Theta^\top$ for $\calM = {\rm Gr}(n,d)$). One can validate the approximations \eqref{eq:lowrank-grad} and \eqref{eq:RFIM-kron} for the low-rank matrix completion, the low-dimension subspace learning, and the deep learning with normalization layers by following \citep{martens2015optimizing,grosse2016kronecker,hu2022riemannian}. 
% In particular, we make the following assumption. 
% \begin{assum} \label{assum:grad-lowrank}
% For \eqref{prob}, we assume the gradient of a single sample has a low-rank approximation, i.e., $\nabla \varphi(\Theta;x,y) \approx \log p(y|f(x,\Theta))$ takes the form
% where $B(\Theta;x,y) \in \R^{n\times r}$ and $A(\Theta; x,y) \in \R^{d \times r}$ with a positive integer $r \leq \max(n,d)$. Moreover, the orthogonal projection onto the tangent space $P_{T_{\Theta} \calM}(\cdot)$ has the form
% \[ P_{T_{\Theta} \calM}(U) = Q(\Theta)U,  \]
% where $Q(\Theta) \in \R^{n\times n}$.
% \end{assum}
% \begin{remark}
% As shown later, the above assumption will be satisfied by the low-rank matrix completion and the low-dimension subspace learning problem. For the Grassmann manifold, the associated $Q(\Theta)$ is $I - \Theta \Theta^\top$. 
% \end{remark}
% With Assumption \ref{assum:grad-lowrank} and by following \citep[Section 4]{hu2022riemannian}, 
 Similarly, by denoting $A_j(\Theta) = A(\Theta; x_j,y_j)$, the \revise{empirical RFIM} has the approximation
\be \label{eq:REFIM-kron} \bar{F}(\theta) \approx A^E(\theta) \otimes [Q(\Theta) B^E(\theta) Q(\Theta)^\top] \ee
with $A^E(\theta):= \frac{1}{N}\sum_{i=1}^N \left\{ \frac{1}{|\mathcal{D}_i|} \sum_{j \in D_i}[A_j(\Theta)A_j(\Theta)^\top]\right\} \in  \\  \R^{d \times d}$ and $B^E(\theta):= \frac{1}{N}\sum_{i=1}^N $ $\left\{ \frac{1}{|\mathcal{D}_i|} \sum_{j \in \mathcal{D}_i}  [B_j(\Theta) \right. $ $\left.B_j(\Theta)^\top]  \right\} \in \R^{n\times n}$. 
% Hence, both RFIM can Riemannian EFIM has the Kronecker-product approximations under the decomposition \eqref{eq:lowrank-grad}. 
Such Kronecker-product approximations are crucial to developing a communication-effective Riemannian natural gradient method in the decentralized setting.

Let us explain how to perform the communication of RFIM based on the Kronecker-product approximation \eqref{eq:RFIM-kron}. It is also applicable to \eqref{eq:REFIM-kron}. With \eqref{eq:RFIM-kron}, one can communicate the $d$-by-$d$ and the $n$-by-$n$ Kronecker factors, i.e., $A^F(\theta)$ and $Q(\Theta)B^F(\Theta)Q(\Theta)$, instead of $F(\theta)$ itself. To be specific, in the $k$-th iteration, 
\begin{align}
\label{eq:fim-tracking-1} 
\bar{A}_{i,k} & = \sum_{j=1}^N w^t_{ij} \bar{A}_{j,k-1} + \hat{A}_{i,k} - \hat{A}_{i,k-1}, \\
\label{eq:fim-tracking-2} 
\bar{B}_{i,k} & = \sum_{j=1}^N w^t_{ij} \bar{B}_{j,k-1} + \hat{B}_{i,k} - \hat{B}_{i,k-1},
\end{align}
where $\hat{A}_{i,k}:= \frac{1}{|\mathcal{D}_i|}\sum_{j \in \mathcal{D}_i} \mathbb{E}_{p(y|f(x_j,\theta_{i,k}))} [A(\Theta_{j,k}; x_j, y) \\ A(\Theta_{j,k}; x_j, y)^\top] \in \R^{d\times d}$, $\hat{B}_{i,k}:= \frac{1}{|\mathcal{D}_i|}\sum_{j \in \mathcal{D}_i}$ $ \mathbb{E}_{p(y|f(x_j,\theta))}[B(\Theta_{j,k};x_j,y) B(\Theta_{j,k};x_j, y)^\top] \in \R^{n\times n}$, $\bar{A}_{i,0} := \hat{A}_{i,0}$, and $\bar{B}_{i,0} = \hat{B}_{i,0}$. We note that when $n \approx d$, the communication costs in \eqref{eq:fim-tracking-1} and \eqref{eq:fim-tracking-2} are approximately twice as those in the gradient tracking \eqref{eq:gt}. If $d \ll n$, we can only track $\bar{A}_i$ as in \eqref{eq:fim-tracking-1} and set $\bar{B}_{i,k} = \hat{B}_{i,k}$. One can only track $\bar{B}_i$ as in \eqref{eq:fim-tracking-2} and always set $\bar{A}_{i,k} = \hat{A}_{i,k}$ for the case of $n \ll d$. \revise{Here, ``$d\ll n$'' can be understood as $d \leq \kappa n$ with a small $\kappa$, say $\kappa =0.1$. The condition ``$d\ll n$'' or ``$n\ll d$'' is used to control the trade-off between the communication cost and the quality of the constructed RFIM. } \revise{ Let $|E|$ be the number of edges of the graph associated to $W$. For both cases, the communication cost in each iteration is $\mathcal{O}(|E|t\min\{n^2, d^2\})$, which is much less than those in the gradient tracking \eqref{eq:gt}, i.e., $\mathcal{O}(|E|tnd)$.}

% \revise{For the gradient tracking, one needs to perform $\mathcal{O}(|E|tnd)$ communication in each iteration, where $|E|$ is the number of edges of the graph associated to $W$. For our proposed DRNGD, the communication costs for RFIM in each iteration are $\mathcal{O}(|E|t(n^2+ d^2))$. Hence, when $n \approx d$, the communication costs for RFIM are approximately twice as those in the gradient tracking. For the cases of $d \ll n$ and $n \ll d$, the communication cost for RFIM, i.e., $\mathcal{O}(|E|t\min\{n^2, d^2\})$,  is much less than those in the gradient tracking.}

\subsection{Our algorithm: DRNGD}

% \revise{This notation is confusing, define $\hat{A}_j$}, the Riemannian EFIM can be written as
% \be \label{eq:fim}
% \begin{aligned}
% \bar{F}(\theta) &  = \mathbb{E}_N \grad \varphi_j(\theta) \grad \varphi_j(\theta)^\top \\
% & = \left( \mathbb{E}_N A_jA_j^\top\right) \otimes \left( \mathbb{E}_N Q(\Theta) B_j B_j^\top Q(\Theta)^\top \right).
% \end{aligned}
% \ee

% Since the Fisher information matrix $F$ lies in $\R^{(mn) \times (mn)}$, direct communication on $F$ is generally not affordable. 
% \revise{Emphaise this assumption, summarize the assumptions on the Euclidean gradient (3) and the Riemannian gradient. Change title? LRMC and LRSL Problems with this structure } 

Note that $\bar{A}_{i,k}$ and $\bar{B}_{i,k}$ may not be positive semidefinite. Let $\lambda^A_{i,k}$ and $\lambda^B_{i,k}$ be the smallest eigenvalues of $\bar{A}_{i,k}$ and $Q(\Theta_{i,k})\bar{B}_{i,k}Q(\Theta_{i,k})$, respectively. To ensure the positive-definite property, we construct the RFIM for the $i$-th agent as follows:
\be \label{eq:local-FIM} 
 F_{i,k} =\begin{cases}
\bar{A}_{i,k} \otimes \check{B}_{i,k}, \; &{\rm if}~ \lambda^A_{i,k}, \lambda^B_{i,k} \geq \sqrt{\lambda}, \\
\tilde{A}_{i,k} \otimes \check{B}_{i,k}, \; & {\rm if}~ \lambda^A_{i,k} < \sqrt{\lambda} \leq \lambda^B_{i,k}, \\
\bar{A}_{i,k} \otimes \tilde{B}_{i,k}, \;& {\rm if}~  \lambda^B_{i,k} < \sqrt{\lambda} \leq \lambda^A_{i,k}, \\
\tilde{A}_{i,k} \otimes \tilde{B}_{i,k}, \; &{\rm if}~ \lambda^A_{i,k}, \lambda^B_{i,k} < \sqrt{\lambda},
\end{cases} \ee
where $\lambda > 0$ is small constant, $\tilde{A}_{i,k} = \hat{A}_{i,k} + \sqrt{\lambda} I$, $\tilde{B}_{i,k} = Q(\Theta_{i,k})\hat{B}_{i,k} Q(\Theta_{i,k})^\top + \sqrt{\lambda} I$, and $\check{B}_{i,k} = Q(\Theta_{i,k})\bar{B}_{i,k} Q(\Theta_{i,k})^\top$.
% To ensure the positive-definite
% property of the constructed local FIM, we compute the 
% Once we obtain $\bar{A}_{i,k}$ and $\bar{B}_{i,k}$, the local FIM is defined as
% \be \label{eq:local-FIM} F_{i,k} = \bar{A}_{i,k} \otimes (Q(\Theta_{i,k})\bar{B}_{i,k} Q(\Theta_{i,k})^\top).  \ee
Then, the local natural gradient direction is given by 
\be \label{eq:ngd} d_{i,k} = P_{T_{\theta_{i,k}} \calM} (F_{i,k}^{-1} y_{i,k}). \ee
Correspondingly, the local update at the $i$-th agent is  \be \label{eq:dngd} \theta_{i,k + 1} = R_{\theta_{i,k}}\left( \alpha P_{T_{\theta_{i,k}}\calM}\left(\sum_{j=1}^N w^t_{ij}\theta_{j,k} \right) -  \beta d_{i,k} \right).  \ee
% Since $F_{i,k}$ may not be positive definite, the ascent property of $\tilde{d}_{i,k}$, i.e., 
% \be \label{eq:ascent-cond} \gamma_1 \|v_{i,k}\|^2 \leq \tilde{d}_{i,k}^\top v_{i,k}, \quad \|\tilde{d}_{i,k}^\top v_{i,k}\| \leq \gamma_2 \|v_{i,k}\|\ee 
% with $\gamma_2 > 1 > \gamma_1 > 0$, may not hold. This also happen for the Newton method \citep[Subsection 3.4]{nocedal1999numerical}. If the ascent condition \eqref{eq:ascent-cond} is satisfied by $\tilde{d}_{i,k}$, we set $d_{i,k} = \tilde{d}_{i,k}$. Otherwise, we set $d_{i,k} = \zeta v_{i,k}$ with a positive $\gamma_1 < \zeta \leq 1 < \gamma_2$. In both cases, the descent property \eqref{eq:ascent-cond} holds for $d_{i,k}$. 
We present our DRNGD method in Algorithm \ref{alg:drngd}. In summary, in each iteration, we perform both gradient tracking and RFIM tracking. Then, a positive-definite local RFIM is defined based on the rule given by \eqref{eq:local-FIM}. Consequently, we have the local natural gradient direction and do the local Riemannian natural gradient update \eqref{eq:dngd}. 
\begin{algorithm}[H] 
    \caption{A decentralized Riemannian natural gradient descent (DRNGD) method for solving \eqref{prob}.}
    \begin{algorithmic}[1] 
	\STATE \textbf{Input: }{Initial point $\theta_0 \in \calM$, integer $t \geq 1$,  damping parameter $\lambda > 0$, $\alpha > 0$, and fixed step size $\beta$. Set $k=0$, $\theta_{i,0} = \theta_0$, $y_{i,0} = \frac{1}{|\mathcal{D}_i|} \sum_{j \in \mathcal{D}_i} \nabla \varphi_j(\theta_{i,k})$, $\bar{A}_{i,0} = \hat{A}_{i,0}$, and $\bar{B}_{i,0} = \hat{B}_{i,0}$ for all $i \in [N]$.}
	
	\FOR{$k= 0, 1,\ldots, K$}
	    \FOR{$i=1,2,\ldots, N$}
	    \STATE Compute $v_{i,k} = P_{T_{\theta_{i,k}} \calM}(y_{i,k})$.
	    
	    \STATE Compute the approximate RFIM $F_{i,k}$ by \eqref{eq:local-FIM} and $d_{i,k}$ as in \eqref{eq:ngd}.
	    
	    \STATE Update $\theta_{i,k+1}$ as in \eqref{eq:dngd}.
	    
	    \STATE Do Riemannian gradient tracking:
	    \[ y_{i,k+1} = \sum_{j=1}^N w^t_{ij} y_{j,k} + g_{i,k+1} - g_{i,k}. \]
	    
	    \STATE Track the RFIM:
	    
	    \IF{$n \ll d$}
	    
	    \STATE Update $\bar{A}_{i,k}$ as in \eqref{eq:fim-tracking-1} and set $\bar{B}_{i,k} = \hat{B}_{i,k}$.
	    
	    \ELSIF{$d \ll n$}
	        
	        \STATE Update $\bar{B}_{i,k}$ as in \eqref{eq:fim-tracking-2} and set $\bar{A}_{i,k} = \hat{A}_{i,k}$.
	    
	    \ELSE
	    \STATE Update $\bar{A}_{i,k}$ and $\bar{B}_{i,k}$ as in \eqref{eq:fim-tracking-1} and \eqref{eq:fim-tracking-2}, respectively.
	    \ENDIF
	    \ENDFOR
	    \ENDFOR
	 \end{algorithmic}
	 \label{alg:drngd}
% 	\caption{A decentralized Riemannian natural gradient for solving \eqref{prob}.}
\end{algorithm}
\subsection{Practical DRNGD for low-rank matrix completion and low-dimension subspace learning}
In this part, we will present two applications of form \eqref{prob}, whose RFIM and Riemannian EFIM yield the Kronecker-product approximations with $d \ll n$.

\textbf{Low-rank matrix completion.}
The goal of the low-rank matrix completion (LRMC) problem is to recover a low-rank matrix from an observed matrix $X\in\mathbb{R}^{n\times T}$. Let $\Omega$ be the set of indices of known entries in $X$, the rank-$d$ LRMC problem has the following form:
\begin{equation}\label{prob:lrmc-org}
    \min_{U\in \mathrm{Gr}(n,d), V\in\mathbb{R}^{d\times T}}\frac{1}{2T}\|\mathcal{P}_{\Omega}(UV - X)\|^2,
\end{equation}
where the operator $\mathcal{P}_{\Omega}$ is defined in an element-wise manner with $(\mathcal{P}_{\Omega}(X))_{ij} = X_{ij}$ if $(i,j)\in \Omega$ and 0 otherwise. As in \cite{hu2022riemannian}, define $v_i$ as 
\[ v_i = v(U; x_i) = \argmin_{v \in \R^d} \| P_{\Omega_{x_i}} (Uv - x_i) \|, \]
where $x_i$ is the $i$-th column of $X$ and the $j$-th element of $P_{\Omega_{x_i}}(s)$ is $s_j$ if $(j,i) \in \Omega$ and $0$ otherwise. Consider the case when the columns $\{x_j\}_{j=1}^N$ are partioned into $N$ agents. The $i$-th agent is with the dataset $\{x_j\}_{j\in \mathcal{D}_i}$ and $\mathcal{D}_1 \cup \cdots \cup \mathcal{D}_N = [T]$.  
The LRMC problem can be reformulated as 
\be \label{prob:lrmc} \min_{U\in \mathrm{Gr}(n,d)}\frac{1}{2N} \sum_{i=1}^N \left[ \frac{1}{|\mathcal{D}_i|} \sum_{j \in \mathcal{D}_i} \|\mathcal{P}_{\Omega_{x_j}}(Uv(U;x_j) - x_j)\|^2 \right]. \ee

By following the analysis in \citep[Subsection 4.2.1]{hu2022riemannian}, the RFIM can be efficiently approximated by 
\[ \left(\frac{1}{N}\sum_{i=1}^N \left[ \frac{1}{|\mathcal{D}_i|} \sum_{j \in \mathcal{D}_i} v(U; x_j) v(U; x_j)^\top \right] \right) \otimes (I-UU^\top).   \]
Note that the projection matrix $Q(U)$ for the Grassmann manifold is $Q(U) = I - UU^\top$. We further have $\hat{A}_i = \frac{1}{|\mathcal{D}_i|} \sum_{j \in \mathcal{D}_i} v(U; x_j) v(U; x_j)^\top $ and $\hat{B}_i = I$.  
% By comparing with \eqref{eq:RFIM-kron}, we have $A^F(u) = \frac{1}{N} \sum_{i=1}^N v(U; x_i) v(U; x_i)^\top$ and $B^F(u) = I$. 
Hence, we only need to communicate the $d$-by-$d$ matrix $v(U; x_i) v(U; x_i)^\top$. The inverse of $F_{i,k}$ can be obtained via computing the inverse of a $d$-by-$d$ matrix, which will be extremely cheap when the rank $d$ is smaller.

\textbf{Low-dimension subspace learning.}
In multi-task learning \cite{ando2005framework,mishra2019riemannian},
different tasks are assumed to share the same latent low-dimensional feature
representation. Specifically, suppose that the $i$-th task  has the 
training set $X_i\in \R^{d_i\times n}$ and the corresponding label set $y_i\in
\R^{d_i}$ for $i=1,\ldots, T$. Let $\{X_j, y_j\}_{i \in \mathcal{D}_i}$ be the local dataset of the $i$-th agent and $\mathcal{D}_1 \cup \cdots \cup \mathcal{D}_N = [T]$.
The multi-task feature learning problem can then be formulated as
\begin{equation}
	\label{prob:subspace-learning}
	\min_{U \in \text{Gr}(n,d)}  \frac{1}{2N}\sum_{i=1}^N \left[ \frac{1}{|\mathcal{D}_i|} \sum_{j \in \mathcal{D}_i} \|X_j U v(U; X_j, y_j) - y_j\|^2 \right],
\end{equation}
where $v(U; X_j, y_j) = \arg\min_{v}  \frac{1}{2} \|X_j U v - y_j\|^2 + \mu \|v\|^2$ and $\mu > 0$ is a regularization parameter.

Following the analysis in \citep[Subsection 4.2.2]{hu2022riemannian}, we have 
\[ \begin{aligned}
\hat{A}_i & = \frac{1}{|\mathcal{D}_i|}  \sum_{j \in \mathcal{D}_i} v_j v_j^\top,\\
 \hat{B}_i & = \frac{1}{|\mathcal{D}_i|} \sum_{j \in \mathcal{D}_i} (I-UU^\top) X_jX_j^\top (I - UU^\top) . 
\end{aligned}
  \]
The matrix $\hat{A}_i$ is $d$-by-$d$ and can be communicated with low costs and we do not communicate $X_iX_i^\top$ due to its high dimensionality as well as privacy concerns.

\subsection{Convergence analysis} \label{subsec:con}

\begin{assum}
	\label{assum:graph}
	The symmetric mixing matrix $W$ is doubly stochastic and satisfies $W_{i j} \geq 0$ and $1>W_{i i}>0$ for all $i, j \in [N]$. The second largest singular value $\sigma_2$ belongs to $(0,1]$. 
\end{assum}

\begin{assum} \label{assum:obj}
The losses $\varphi(\Theta;x,y)$ are continuously differentiable and with $L_g$-Lipschitz continuous gradient. In addition, the mappings $A(\Theta; x,y)$ and $B(\Theta; x, y)$ are also Lipschitz continuous with moduli $L_a$ and $L_b$. The manifold $\calM$ is the Stiefel manifold ${\rm St}(n,d)$ or the Grassmann manifold ${\rm Gr}(n,d)$. The retraction $R$ used in Algorithm \ref{alg:drngd} is the polar decomposition. 

% and with $L_p$-Lipschitz continuous $A_i(\Theta)$ and $B_i(\Theta)$ over ${\rm St}(n,d)$, i.e., for any $i \in [N]$ and $\Theta_1, \Theta_2 \in {\rm St}(n,d)$, 
% \be \label{eq:lip-ag} \begin{aligned}
% & \max\{\| A_i(\Theta_1) - A_i(\Theta_2) \|, \| \| B_i(\Theta_1) - B_i(\Theta_2)\| \} \\
% \leq& L_p\|\Theta_1 - \Theta_2\|.
% \end{aligned}
%  \ee
\end{assum}

With the above assumption, we denote $C_g:= \max_{\Theta \in {\rm St}(n,d)} \; \|\nabla \varphi(\Theta; x, y)\|$, $C_a := \max_{\Theta \in {\rm St}(n,d)}$ $\| A(\Theta; x, y)\|$, and $C_b:=\max_{\Theta \in {\rm St}(n,d)} \; \|B(\Theta; x, y)\|$. To prove the convergence, we need to establish the following Riemannian upper bound inequality and the Lipschitz continuity of the Riemannian gradient on the Grassmann manifold. The proof is given in Appendix \ref{append:lemma}. We note that similar conclusions on the Stiefel manifold are investigated in \citep{chen2021decentralized}. 
\begin{lem} \label{lem:upper-bound-grass}
Suppose that $\calM$ is the Grassmann manifold ${\rm Gr}(n,d)$ and  Assumption \ref{assum:obj} holds, we have
\be \label{eq:upbound-grass} \varphi(U) \leq \varphi(V) + \iprod{\grad \varphi(V)}{U-V} + \frac{\hat{L}}{2} \|U-V\|^2, \ee
for any $U,V\in {\rm St}(n,d)$ and $\hat{L} = L_g + C_g$. In addition, the Riemannian gradient $\grad \varphi(V) = (I-VV^\top) \nabla \varphi(V)$ is Lipschitz continuous with modulus $\tilde{L} = L_g + 2C_g$. 
\end{lem}

% \begin{lem} \label{lem:consensusbound}
% 	Assume that Assumption \ref{assum: Connected graph} holds.
% 	Let $\alpha \in (0, \bar{\alpha}]$ with some $\bar{\alpha}\in (0,1]$, $0 < \beta_k \leq \min\{\frac{1-\rho_t}{L} \delta_1, \frac{\alpha \delta_1}{5L} \})$, and $t \geq \lceil \log_{\sigma_2}(\frac{1}{2\sqrt{n}}) \rceil$. If $\vx_0 \in \calN$, then $\vx_k \in \calN$ for all $k \geq 0$ generated by the DRNGD algorithm. Moreover, we have 
% 	\begin{align}\label{eq: consensusboundgeneralbeta}
% 		& \quad \| \vx_{k+1} - \bar{\vx}_{k+1} \|_F \nonumber \\
% 		& \leq  \rho_t^{k+1} \|\vx_0 - \bar{\vx}_0 \|_F + \sqrt{n}L \sum_{l=0}^k \rho_t^{k-l}\beta_l.
% 	\end{align}
% \end{lem}

Before presenting our convergence theorem, we introduce the following notations. Let $\mathbf{\Theta}_k:= [\Theta_{1,k}^\top; \Theta_{2,k}^\top, \ldots, \Theta_{N,k}^\top]^\top \in \R^{(Nn) \times d}$ and $\bar{\mathbf{\Theta}}_k = \mathbf{1}_N \otimes \bar{\Theta}_k$ with \revise{an} $N$-dimension vector $\mathbf{1}_N$ of all entries $1$. As the linear consensus on the Stiefel manifold relies on the initial consensus error \citep{chen2021local}, we define the neighborhood $\mathcal{N} := \mathcal{N}_1 \cap \mathcal{N}_2$ with $\mathcal{N}_1:=\{\mathbf{\Theta}: \|\mathbf{\Theta} - \bar{\mathbf{\Theta}}\|^2 \leq N \delta_1^2 \}$ and $\mathcal{N}_2 := \{ \mathbf{\Theta}: \max_{i \in [N]} \|\Theta - \bar{\Theta}\| \leq \revise{\delta_2}\}$ with $\delta_1, \delta_2$ satisfying $\delta_1 \leq \frac{1}{5\sqrt{d}} \delta_2$ and $\delta_2 \leq \frac{1}{6}$. We have the following result to a stationary point with proofs in Appendix \ref{append:thm}.
\begin{thm} \label{thm:convergence}
 Suppose that Assumptions \ref{assum:graph}, and \ref{assum:obj} hold. Let $\mathbf{\Theta}_0 \in \mathcal{N}, t \geq\left\lceil\log _{\sigma_2}\left(\frac{1}{2 \sqrt{N}}\right)\right\rceil, 0<\alpha \leq \bar{\alpha}$ for some $\bar{\alpha} \in (0,1]$, and $\beta \in (0, \bar{\beta})$ with $\bar{\beta} > 0$ depending on $\sigma_2, t, L_G, C_g, C_a, C_b, \delta_1, \delta_2$, $\alpha$, $\gamma_1$, and $\gamma_2$. 
Then, for the sequences generated by Algorithm \ref{alg:drngd}
$$
\begin{aligned}
& \min _{k \leq K} \frac{1}{N}\left\|\mathbf{\Theta}_k-\bar{\mathbf{\Theta}}_k\right\|^2 \leq \mathcal{O}(\frac{\beta}{K}), \\
& \min_{k \leq K}\left\|\grad \varphi\left(\bar{\Theta}_k\right)\right\|^2 \leq \mathcal{O}(\frac{1}{\beta K}),
\end{aligned}
$$
where the constants in $\mathcal{O}$ are related to $\alpha$, $\sigma_2, t$, $\varphi(\bar{\Theta}_0),$ $\inf_{\Theta \in \calM} \varphi(\Theta)$,  $L_g, C_g, L_a, L_b, C_a, C_b, \lambda, \delta_1$, and $\delta_2$. 
\end{thm}

\begin{remark}\revise{
Our analysis can deal with the cases of $\calM$ being either the Stiefel manifold or the Grassmann manifold, while the analysis in \citep{chen2021decentralized} is only applicable to the Stiefel manifold case. For the case of the Grassmann manifold,  the key is to establish the Riemannian upper bound condition and Lipschitz continuity of the Riemannian gradient in Lemma \ref{lem:upper-bound-grass}.}

\revise{
It is worth mentioning that the convergence result in Theorem \ref{thm:convergence} is established with respect to a natural gradient direction instead of the negative gradient direction used in \citep{chen2021decentralized}. Our main effort is to show the descent properties of the used natural gradient descent direction, including the positive definiteness of the preconditioned matrix (i.e., $F_{i,k} + \lambda I$) and the boundedness of $\| F_{i,k} \|$. Then, we establish the global convergence with respect to a general descent direction (including the method in \citep{chen2021decentralized} as a special case), which could serve as a basis for analyzing other decentralized Riemannian second-order methods. }
\end{remark}
\begin{figure*}[!htb]
	\begin{center}
		\begin{minipage}[b]{0.30\linewidth}
			\centering
			\centerline{\includegraphics[width=\linewidth]{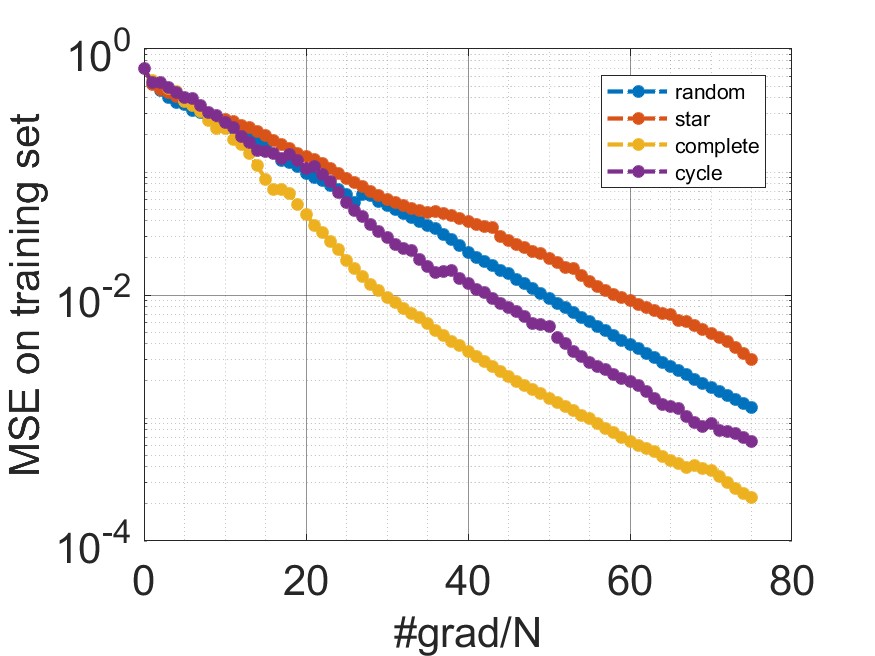}}
			%  \vspace{1.0cm}
			\centerline{}\medskip
		\end{minipage}
		\hspace{0.2cm}
		\begin{minipage}[b]{0.30\linewidth}
			\centering
			\centerline{\includegraphics[width=\linewidth]{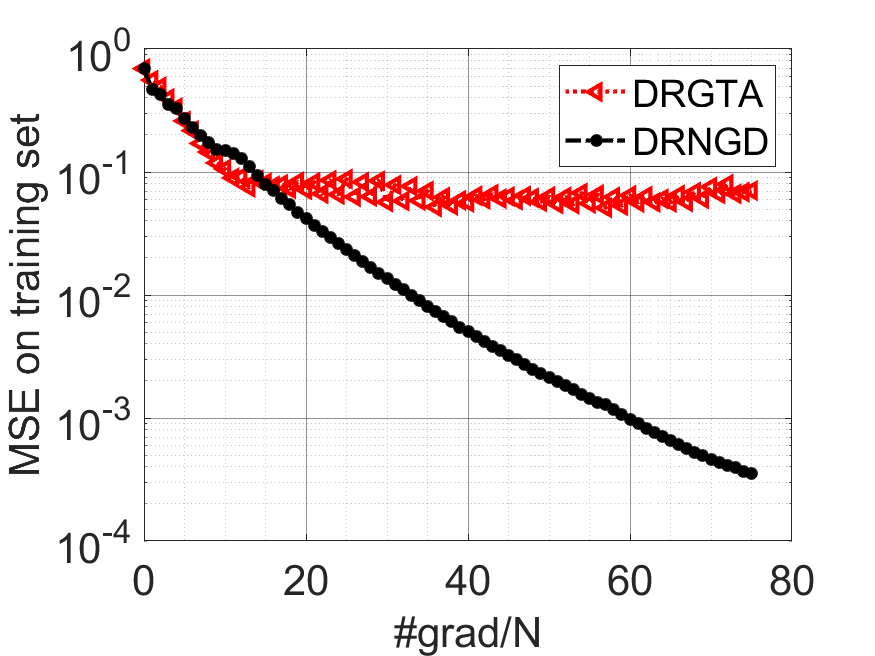}}
			%  \vspace{1.0cm}
			\centerline{}\medskip
		\end{minipage}
		\hspace{0.2cm}
		\begin{minipage}[b]{0.30\linewidth}
			\centering
			\centerline{\includegraphics[width=\linewidth]{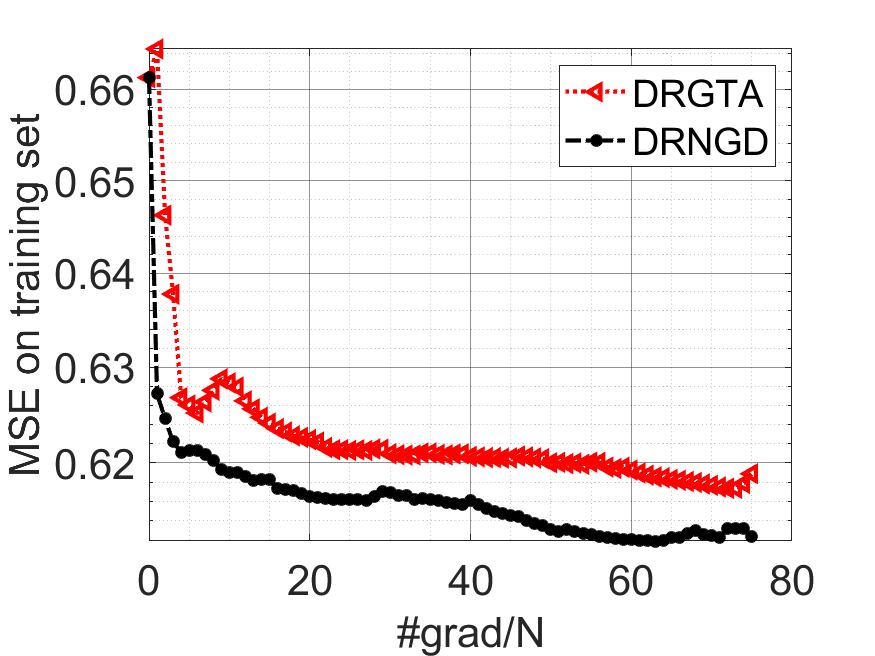}}
			%  \vspace{1.0cm}
			\centerline{}\medskip
		\end{minipage}
		%\hfill
		\\
		\vskip -0.2in
		%\hfill
		\begin{minipage}[b]{0.30\linewidth}
			\centering
			\centerline{\includegraphics[width=\linewidth]{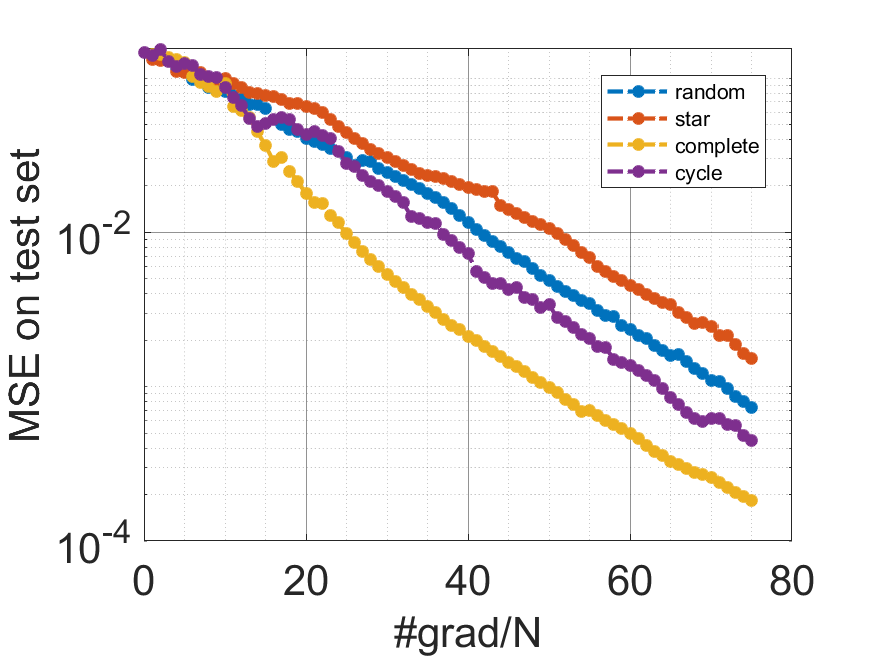}}
			%\cspace{1.5cm}
			\centerline{} \medskip
			\vskip -0.3in
			\caption{Convergence performance of DRNGD for the subspace learning problem with different graphs.}
			\label{fig:sl-graph}
		\end{minipage}
		\hspace{0.2cm}
		\begin{minipage}[b]{0.30\linewidth}
			\centering
			\centerline{\includegraphics[width=\linewidth]{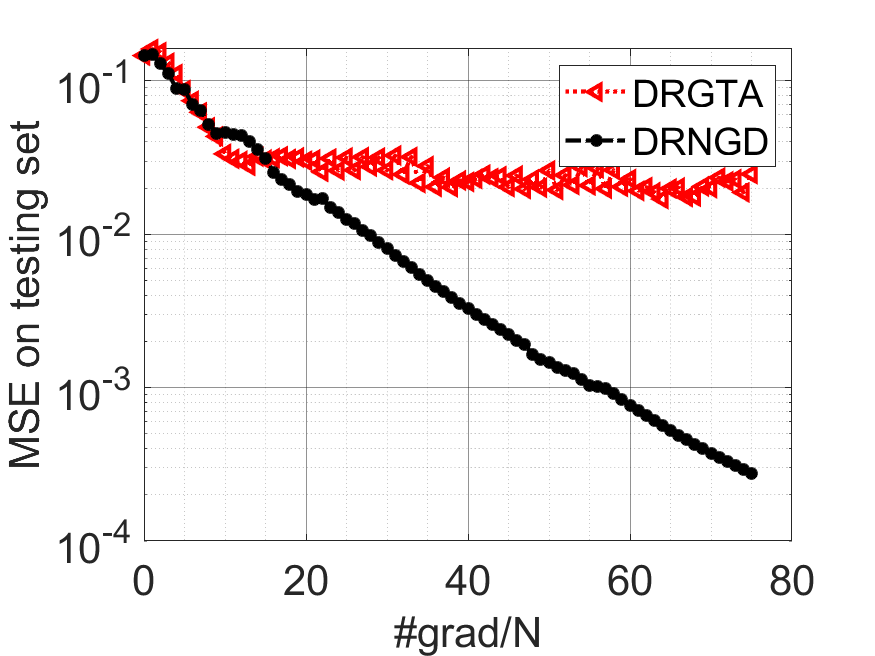}}
			%\cspace{1.5cm}
			\centerline{} \medskip
			\vskip -0.3in
			\caption{Convergence performance of DRNGD for the subspace learning problem with random dataset.}
			\label{fig:sl-random}
		\end{minipage}
		\hspace{0.2cm}
		\begin{minipage}[b]{0.30\linewidth}
			\centering
			\centerline{\includegraphics[width=\linewidth]{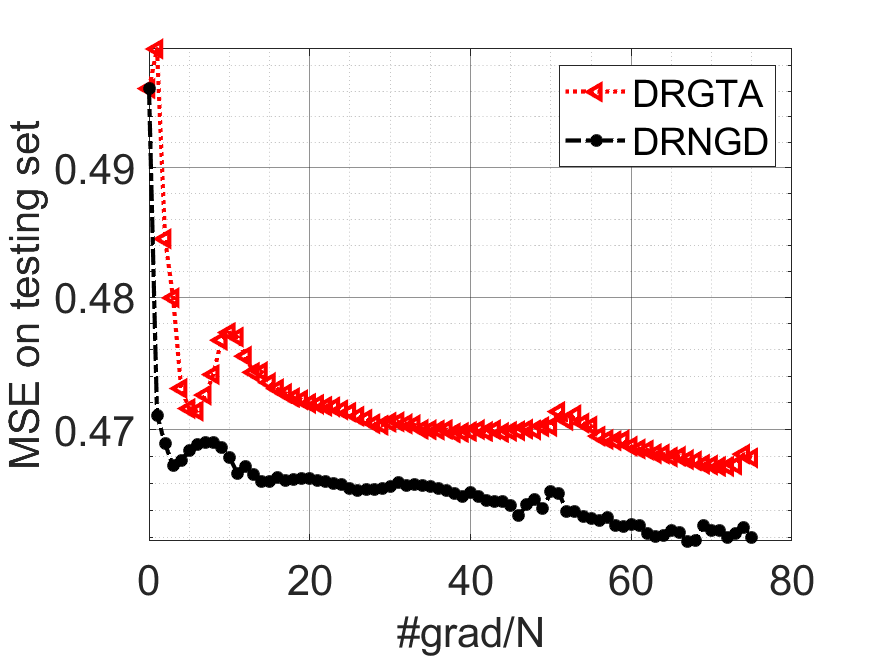}}
			%\cspace{1.5cm}
			\centerline{} \medskip
			\vskip -0.3in
			\caption{Convergence performance of DRNGD for the subspace learning problem with School dataset.}
			\label{fig:sl-school}
		\end{minipage}
	\end{center}
	
% 	\label{fig:sl-graph}
\end{figure*}
\section{Numerical experiments}
In this section, our proposed DRNGD in Algorithm \ref{alg:drngd} and the decentralized Riemannian gradient tracking algorithm (DRGTA) proposed in \cite{chen2021decentralized} are compared  on low-dimension subspace learning  and LRMC problems. Our implementations are based on the Manopt toolbox \cite{boumal2014manopt}. For both algorithms, we set $\alpha =1$ and $t=1$ (as different $t$'s lead to similar performances of DRGTA \citep{chen2021decentralized}). The initial point is randomly generated. Fixed step sizes are employed, and a grid search technique is utilized to identify the optimal values. The codes were written in MATLAB and run on a standard PC with 3.00 GHz AMD R5 microprocessor and 16GB of memory.

% \begin{figure*}[!htb]
% 	\begin{center}
% 		\begin{minipage}[b]{0.40\linewidth}
% 			\centering
% 			\centerline{\includegraphics[width=\linewidth]{fig/sl/random/randomtrain-epoch1.png}}
% 			%  \vspace{1.0cm}
% 			\centerline{}\medskip
% 		\end{minipage}
% 		%\hfill
% 		\begin{minipage}[b]{0.40\linewidth}
% 			\centering
% 			\centerline{\includegraphics[width=\linewidth]{fig/sl/random/randomtest-epoch1.png}}
% 			%\cspace{1.5cm}
% 			\centerline{} \medskip
% 		\end{minipage}
% 	\end{center}
% 	\vskip -0.4in
% 	\caption{ Convergence performance of DRNGD for the subspace learning problem with different graphs.}
% 	\label{fig:sl-graph}
% \end{figure*}

% \begin{figure*}[!htb]
% 	\begin{center}
% 		\begin{minipage}[b]{0.40\linewidth}
% 			\centering
% 			\centerline{\includegraphics[width=\linewidth]{fig/sl/random/randomtrain-epoch.png}}
% 			%  \vspace{1.0cm}
% 			\centerline{}\medskip
% 		\end{minipage}
% 		%\hfill
% 		\begin{minipage}[b]{0.40\linewidth}
% 			\centering
% 			\centerline{\includegraphics[width=\linewidth]{fig/sl/random/randomtest-epoch.png}}
% 			%\cspace{1.5cm}
% 			\centerline{} \medskip
% 		\end{minipage}
% 	\end{center}
% 	\vskip -0.4in
% 	\caption{ Convergence performance of DRNGD for the subspace learning problem with random dataset.}
% 	\label{fig:sl-random}
% \end{figure*}

\subsection{Low-dimension subspace learning}

In this subsection, we evaluate the performance of DRNGD on the multi-task learning problem \eqref{prob:subspace-learning} on different benchmarks. 
%To test the performance of our algorithm in different network topologies,  
We generate a random dataset as in \citep{mishra2019riemannian}.  The training instances $X_i$ are generated according to the Gaussian distribution with zero mean and unit standard deviation. A feature subspace $U^*$ for the problem instance is generated as a random point on $U_*\in \mathrm{St}(n,d)$. The weight vector $v_i$ for task $i$ is generated from the Gaussian distribution with zero mean and unit variance. The labels for training instances for task $i$ are computed as $y_i = X_iU_*U_*^Tv_i$. The labels $y_i$ are subsequently perturbed with a random mean zero Gaussian noise with $10^{-6}$ standard deviation.  

\paragraph{The effect of network topology.}
We test the performance of our algorithm in different network topologies:  star graph, ring graph, random graph, and complete graph.  We consider a toy problem instance with agents of $N = 4$. The $i$-th agent has one task and the number of training instances $d_i$ in the task is randomly chosen between 10 and 50. The input space dimension is $n = 100$. The dimension of feature subspace $U^*$ is $d = 6$.  We perform  80/20-train/test partitions.  Figure \ref{fig:sl-graph} shows the training mean square error (MSE) and testing MSE \citep{mishra2019riemannian} of our algorithm  over different network topologies. It can be seen that the performance in complete topology is better than the performance in other topologies. Overall, our algorithm can achieve similar performances in all topologies.

%Figure \ref{fig:sl-random} shows the training MSE and testing MSE of our algorithm and DRGTA  over different network topologies: star topology, ring topology, random topology, and complete topology. Our algorithm achieves better performance in all topologies compared with DRGTA. 
% \begin{figure*}
% 	\centering
% 	\begin{minipage}[b]{0.45\linewidth}
% 		\subfigure[a]{
% 			\includegraphics[width=0.25\textwidth]{fig/sl/randomtest-epoch.png} 
% 			\includegraphics[width=0.25\textwidth]{fig/sl/randomtest-epoch.png}}
%    \subfigure[d]{
% 			\includegraphics[width=0.25\linewidth]{fig/sl/randomtest-epoch.png} 
% 			\includegraphics[width=0.25\textwidth]{fig/sl/randomtest-epoch.png}
% 			\label{fig:hor_4figs_1cap_2subcap_1}
% 	}
% 	\end{minipage}
% 	\caption{c}
% 	\label{fig:hor_4figs_1cap_2subcap}
% \end{figure*}

% \begin{figure*}
% 	\centering
% 	\subfigure[a]{
% 		\begin{minipage}[b]{0.5\textwidth}
% 			\includegraphics[width=0.5\linewidth]{fig/sl/randomtest-epoch.png} 
% 			\includegraphics[width=0.5\linewidth]{fig/sl/randomtest-epoch.png}
% 		\end{minipage}
% 		\label{fig:grid_4figs_1cap_2subcap_1}
% 	}
%     	\subfigure[b]{
%     		\begin{minipage}[b]{0.5\textwidth}
%    		 	\includegraphics[width=0.5\linewidth]{fig/sl/randomtest-epoch.png}
% 		 	\includegraphics[width=0.5\linewidth]{fig/sl/randomtest-epoch.png}
%     		\end{minipage}
% 		\label{fig:grid_4figs_1cap_2subcap_2}
%     	}
% 	\caption{c}
% 	\label{fig:grid_4figs_1cap_2subcap}
% \end{figure*}

%\paragraph{Real-world datasets.}We compare our algorithm and RGTA on two real-world multitask benchmark datasets: Parkinsons and School. 
\begin{figure*}[!htb]
	\begin{center}
		\begin{minipage}[b]{0.22\linewidth}
			\centering
			\centerline{\includegraphics[width=\linewidth]{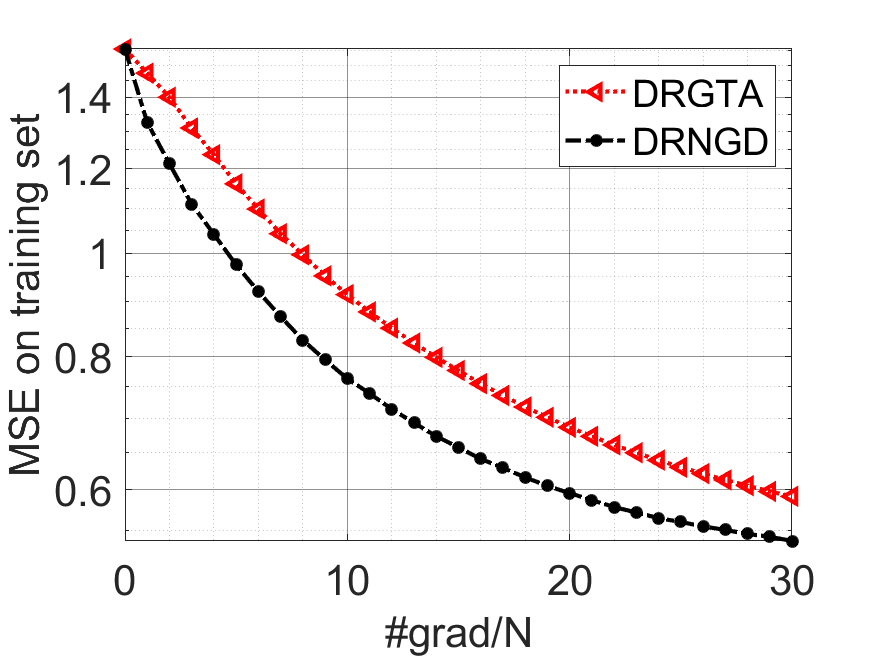}}
			%  \vspace{1.0cm}
			\centerline{}\medskip
		\end{minipage}
		%\hfill
  \begin{minipage}[b]{0.22\linewidth}
			\centering
			\centerline{\includegraphics[width=\linewidth]{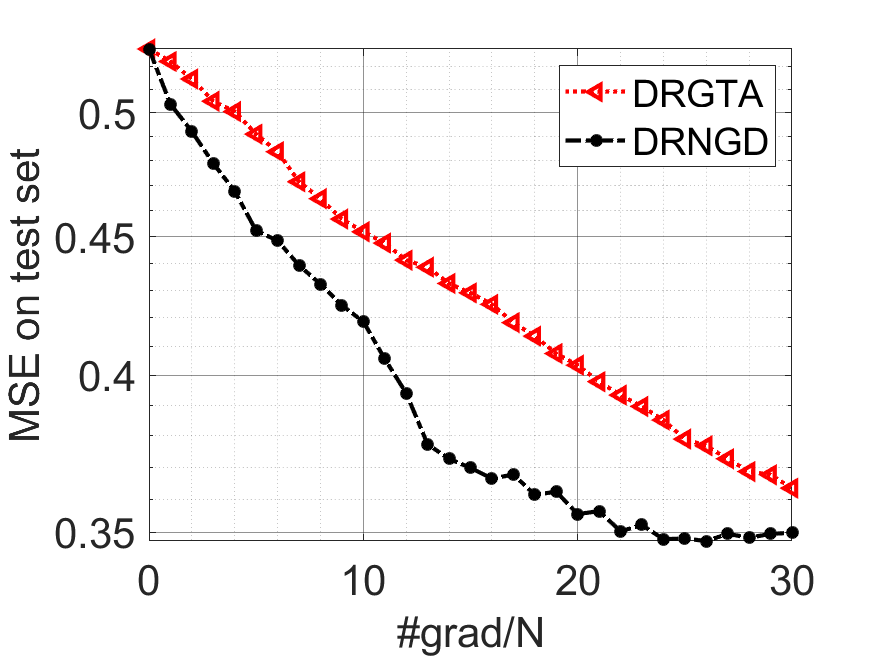}}
			%\cspace{1.5cm}
			\centerline{} \medskip
		\end{minipage}
	\begin{minipage}[b]{0.22\linewidth}
			\centering
			\centerline{\includegraphics[width=\linewidth]{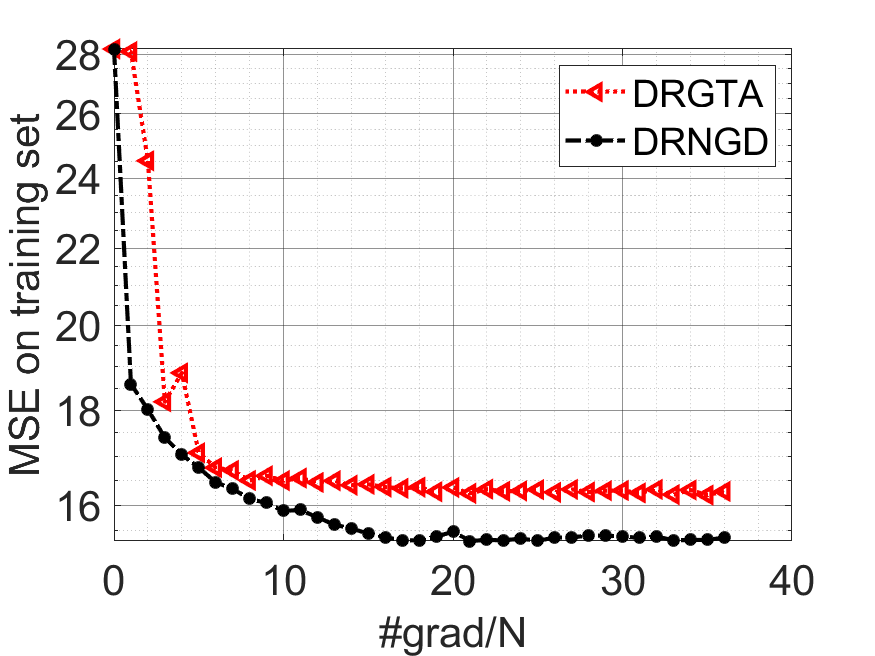}}
			%  \vspace{1.0cm}
			\centerline{}\medskip
		\end{minipage}
		%\hfill
		\begin{minipage}[b]{0.22\linewidth}
			\centering
			\centerline{\includegraphics[width=\linewidth]{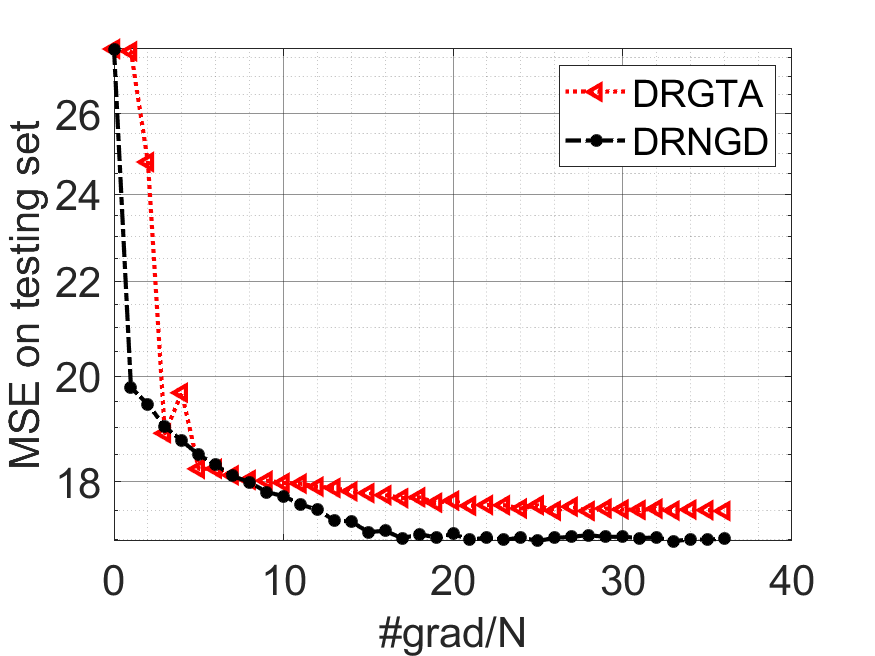}}
			%\cspace{1.5cm}
			\centerline{} \medskip
		\end{minipage}
	\end{center}
	\vskip -0.4in
	\caption{ Convergence performance of DRNGD for the LRMC problems. Left: random dataset with $N= d = 2000$, right: Jester dataset.}
	\label{fig:mc:random:N}
\end{figure*}

\begin{figure*}[!htb]
	\begin{center}
		\begin{minipage}[b]{0.22\linewidth}
			\centering
			\centerline{\includegraphics[width=\linewidth]{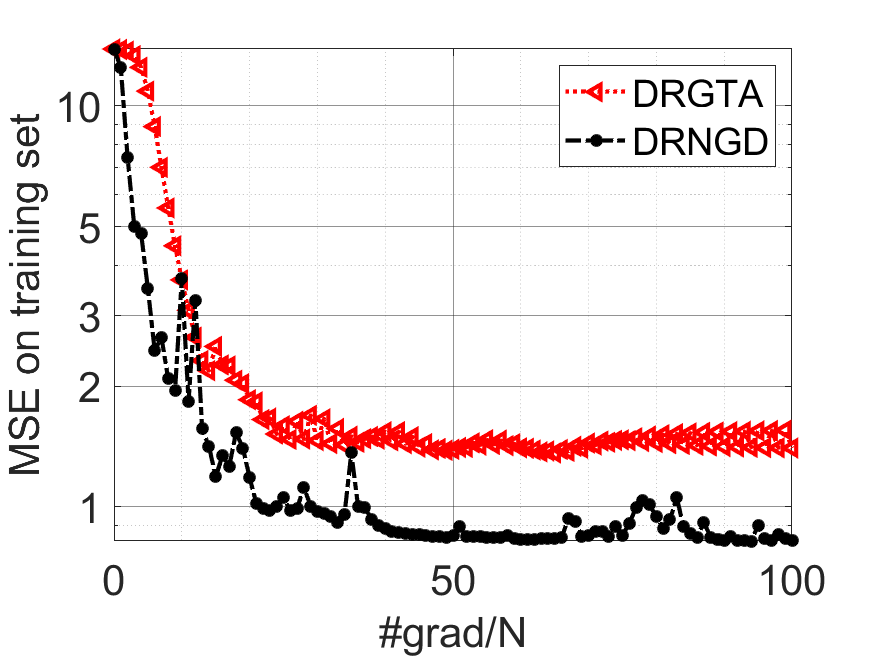}}
			%  \vspace{1.0cm}
			\centerline{}\medskip
		\end{minipage}
		%\hfill
		\begin{minipage}[b]{0.22\linewidth}
			\centering
			\centerline{\includegraphics[width=\linewidth]{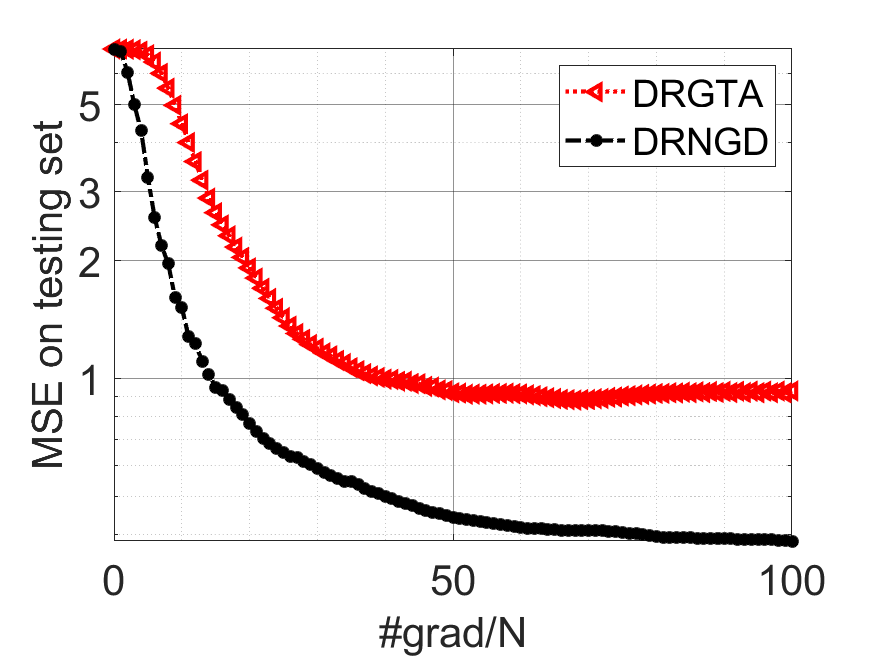}}
			%\cspace{1.5cm}
			\centerline{} \medskip
		\end{minipage}
	\begin{minipage}[b]{0.22\linewidth}
			\centering
			\centerline{\includegraphics[width=\linewidth]{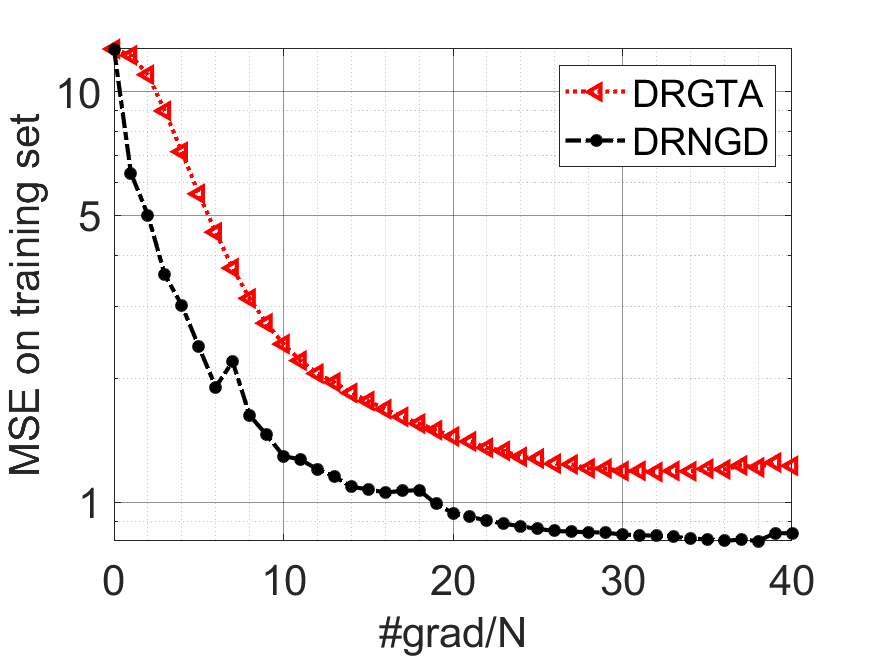}}
			%  \vspace{1.0cm}
			\centerline{}\medskip
		\end{minipage}
		%\hfill
		\begin{minipage}[b]{0.22\linewidth}
			\centering
			\centerline{\includegraphics[width=\linewidth]{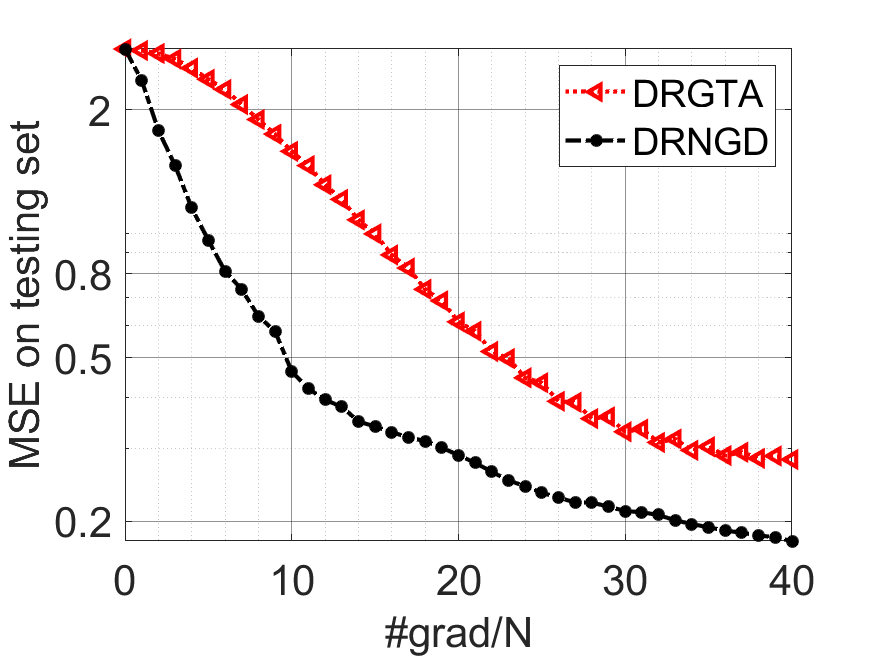}}
			%\cspace{1.5cm}
			\centerline{} \medskip
		\end{minipage}
	\end{center}
	\vskip -0.4in
	\caption{ Convergence performance of DRNGD for the LRMC problem. Left: MovieLens dataset, right: MovieLens100K dataset.}
	\label{fig:mc:ml}
\end{figure*}  

\paragraph{Comparisons with DRGTA.} Figure \ref{fig:sl-random} shows the comparison with DRGTA on the random set generated in the same as the previous setting. Here, the network topology is chosen as random topology.   It can be seen that our algorithm achieves better performance compared with DRGTA.

% \begin{figure*}[!htb]
% 	\begin{center}
% 		\begin{minipage}[b]{0.40\linewidth}
% 			\centering
% 			\centerline{\includegraphics[width=\linewidth]{fig/sl/school/schooltrain-epoch.png}}
% 			%  \vspace{1.0cm}
% 			\centerline{}\medskip
% 		\end{minipage}
% 		%\hfill
% 		\begin{minipage}[b]{0.40\linewidth}
% 			\centering
% 			\centerline{\includegraphics[width=\linewidth]{fig/sl/school/schooltest-epoch.png}}
% 			%\cspace{1.5cm}
% 			\centerline{} \medskip
% 		\end{minipage}
% 	\end{center}
% 	\vskip -0.4in
% 	\caption{ Convergence performance of DRNGD for the subspace learning problem with School dataset. 
% % 	\lina{add grids to all the figures; I also think you can make the figures smaller. For instance you can do column arrangement for figure 3, 4, 5. Making the two subfigures in each figure up and down (column). And then from left to right: figure 3, 4, 5. It will save space and make the paper looks nicer.}
% 	}
% 	\label{fig:sl-school}
% \end{figure*}

% \begin{figure*}[!htb]
% 	\begin{center}
% 		\begin{minipage}[b]{0.40\linewidth}
% 			\centering
% 			\centerline{\includegraphics[width=\linewidth]{fig/sl/sarcostrain-epoch.png}}
% 			%  \vspace{1.0cm}
% 			\centerline{}\medskip
% 		\end{minipage}
% 		%\hfill
% 		\begin{minipage}[b]{0.40\linewidth}
% 			\centering
% 			\centerline{\includegraphics[width=\linewidth]{fig/sl/sarcostest-epoch.png}}
% 			%\cspace{1.5cm}
% 			\centerline{} \medskip
% 		\end{minipage}
% 	\end{center}
% 	\vskip -0.25in
% 	\caption{ Convergence performance of DRNGD for the subspace learning problem with ``sarcos'' dataset.}
% 	\label{fig:sl:sarcos}
% \end{figure*}

We also evaluate our algorithm on  two real-world multitask benchmark datasets: School and Sarcos.  %In the Parkinsons dataset, the goal is to predict the Parkinson’s disease symptom score at different times of 42 patients with $n = 19$ biomedical features \cite{jawanpuria2012convex}. A total of 5875 observations are available. The symptom score prediction problem for each patient is considered as a task ($N = 42$).
The School dataset consists of 15362 students
from 139 schools \cite{goldstein1991multilevel,argyriou2008convex}. The aim is to predict the performance (examination score) of the students from the schools, given the description of the schools and past records of the students. A total of $n = 28$ features are given. The examination score prediction problem for each school is considered as a task and there are $T=139$ tasks in total. The Sarcos data \footnote{ \url{http://gaussianprocess.org/gpml/data/}} relates to an inverse dynamics problem for a seven degrees-of-freedom Sarcos anthropomorphic robot arm, which contains 54484 examples with $n = 7$ features. We perform  80/20-train/test partitions. We run DRGND and DRGTA with $N = 4, d = 6$ and randomly split the $T$ tasks into $N$ agents. In Figures \ref{fig:sl:sarcos1} and \ref{fig:sl-school}, we report the performance of our algorithm compared with DRGTA in terms of training MSE and testing MSE. It is clear that our algorithm achieves better performance than DRGTA in all datasets. 

\revise{
For the comparisons with the centralized Riemannian natural gradient method \citep{hu2022riemannian}, we put them in Appendix \ref{sec:app2}. Our algorithm can achieve comparable MSE as RNGD although the convergence speed of our DRNGD is slightly slower than RNGD in terms of \#grad/N. this is due to the inexactness of constructing the Riemannian natural gradient direction. It should be emphasized that the centralized algorithm, RNGD, needs to collect all data into a single agent, which breaks data privacy.}
% \begin{figure*}[htb]
% 	\begin{center}
% 		\begin{minipage}[b]{0.40\linewidth}
% 			\centering
% 			\centerline{\includegraphics[width=\linewidth]{fig/sl/parkinsontrain-epoch.png}}
% 			%  \vspace{1.0cm}
% 			\centerline{}\medskip
% 		\end{minipage}
% 		%\hfill
% 		\begin{minipage}[b]{0.40\linewidth}
% 			\centering
% 			\centerline{\includegraphics[width=\linewidth]{fig/sl/parkinsontest-epoch.png}}
% 			%\cspace{1.5cm}
% 			\centerline{} \medskip
% 		\end{minipage}
% 	\end{center}
% 	\vskip -0.25in
% 	\caption{ Convergence performance of DRNGD for the subspace learning problem with ``Parkinsons'' dataset.}
% 	\label{fig:sl-Parkinsons}
% \end{figure*}

\subsection{Low-rank matrix completion}
% \begin{itemize}
%  \item MSE and consensus error with graph of star, ring, random and complete. 
%     \item Comparison on real-world datasets
% \end{itemize}

In this subsection, we evaluate the performance of DRNGD on the LRMC problem \eqref{prob:lrmc} on different benchmark datasets. The low-rank matrix $X$ is generated via $X = UA$, where $U\in \mathbf{R}^{n\times d}, A\in\mathbb{R}^{d\times T}$ is two random matrices with i.i.d. standard normal distribution. The sample set $\Omega$ is randomly generated on $[n] \times [T]$ with the uniform distribution of $\frac{|\Omega|}{nT}$. The oversampling rate $OS$ for $X$ is defined as
\begin{equation*}
    OS: = \frac{|\Omega|}{(T+n-d)d}.
\end{equation*}
Note that $OS$ represents the difficulty of recovering matrix $X$, and it should be larger than 1.

% \begin{figure*}[!htb]
% 	\begin{center}
% 		\begin{minipage}[b]{0.30\linewidth}
% 			\centering
% 			\centerline{\includegraphics[width=\linewidth]{fig/mc/random/random-train-epoch-1000.png}}
% 			%  \vspace{1.0cm}
% 			\centerline{}\medskip
% 		\end{minipage}
% 		%\hfill
% 		\begin{minipage}[b]{0.30\linewidth}
% 			\centering
% 			\centerline{\includegraphics[width=\linewidth]{fig/mc/random/random-test-epoch-1000.png}}
% 			%\cspace{1.5cm}
% 			\centerline{} \medskip
% 		\end{minipage}
%   \begin{minipage}[b]{0.30\linewidth}
% 			\centering
% 			\centerline{\includegraphics[width=\linewidth]{fig/mc/random/random-consensus-epoch-1000.png}}
% 			%\cspace{1.5cm}
% 			\centerline{} \medskip
% 		\end{minipage}
% 	\end{center}
% 	\vskip -0.25in
% 	\caption{ Convergence performance of DRNGD for the LRMC problem with different topologies $(N = 1000)$. \revise{Add consensus}}
% 	\label{fig:mc:random:graph}
% \end{figure*}

\paragraph{Comparisons with DRGTA}We show comparisons with  \cite{chen2021decentralized}, the decentralized Riemannian gradient tracking algorithm (DRGTA). We consider a problem instance of size  $2000\times 2000$, $N=4$, $d = 5$, and OS $=3$. 
% DRGTA is run with stepsize $\alpha = 0.02$. For RNGD, we set the stepsize $\alpha = 50$. 
As shown in Figure \ref{fig:mc:random:N}, DRNGD outperforms DRGTA in terms of MSE on both the training set and testing set.  

We also evaluate the performance of DRNGD on LRMC with two real-world datasets: the Jester dataset \footnote{\url{http://eigentaste.berkeley.edu/user/index.php.}} and the Movielens dataset \footnote{\url{http://www.movielens.org.}}. The Jester dataset contains 4.1 million ratings for 100 jokes from 73,421 users. For numerical tests, 24,983 users  who have rated 36 or more jokes are chosen. We refer to \cite{ma2011fixed,chen2012matrix} for more details.  The MovieLens2 dataset contains movie rating data from users on different movies. In the following experiments, we choose the dataset
MovieLens100K consists of 100000 ratings from 943 users on 1682 movies, and MovieLens 1M consists of one million movie ratings from 6040 users on 3952 movies.  It can be observed that our DRNGD algorithm achieves better performance compared with DRGTA in both datasets.

% \begin{figure*}[!htb]
% 	\begin{center}
% 		\begin{minipage}[b]{0.49\linewidth}
% 			\centering
% 			\centerline{\includegraphics[width=\linewidth]{fig/mc/ml100/train-epoch.png}}
% 			%  \vspace{1.0cm}
% 			\centerline{}\medskip
% 		\end{minipage}
% 		%\hfill
% 		\begin{minipage}[b]{0.49\linewidth}
% 			\centering
% 			\centerline{\includegraphics[width=\linewidth]{fig/mc/ml100/test-epoch.png}}
% 			%\cspace{1.5cm}
% 			\centerline{} \medskip
% 		\end{minipage}
% 	\end{center}
% 	\vskip -0.25in
% 	\caption{ Convergence performance of DRNGD for the LRMC problem with MovieLens100K dataset.}
% 	\label{fig:mc:ml100}
% \end{figure*}

% \begin{figure*}[!htb]
% 	\begin{center}
% 		\begin{minipage}[b]{0.49\linewidth}
% 			\centering
% 			\centerline{\includegraphics[width=\linewidth]{fig/mc/jester/train-epoch.png}}
% 			%  \vspace{1.0cm}
% 			\centerline{}\medskip
% 		\end{minipage}
% 		%\hfill
% 		\begin{minipage}[b]{0.49\linewidth}
% 			\centering
% 			\centerline{\includegraphics[width=\linewidth]{fig/mc/jester/test-epoch.png}}
% 			%\cspace{1.5cm}
% 			\centerline{} \medskip
% 		\end{minipage}
% 	\end{center}
% 	\vskip -0.25in
% 	\caption{ Convergence performance of DRNGD for the LRMC problem with Jester dataset.}
% 	\label{fig:mc:jester}
% \end{figure*}

\section{Concluding Remarks}\label{sec:con}
For a class of manifold-constrained and decentralized optimization problems, we develop a communication-efficient decentralized Riemannian natural gradient method based on the Kronecker-product approximations of the RFIM and the Riemannian EFIM. The communication cost can be even cheaper than those of Riemannian gradients. In the construction of the local RFIM, we present a switching condition between the local RFIM and the communicated one, which enables us to prove global convergence. Numerical results demonstrate the efficiency of our proposed method compared to the state-of-the-art ones. More experiments on deep learning problems with normalization layers are worth to be further investigated. It would be interesting to consider the generalizations in the federated learning setting.

% \section*{Acknowledgements}
%\newpage 

\bibliographystyle{icml2023}
\bibliography{decen-NGD-manifold.bib}

\newpage
\onecolumn
\begin{appendix}
\begin{center}
{\Large \bf Appendix}

\end{center}
\vspace{-0.1in}
\par\noindent\rule{\textwidth}{1pt}
\setcounter{section}{0}

\renewcommand\thesection{\Alph{section}}
\section{The Stiefel manifold and the Grassmann manifold} \label{append:prem}
As introduced in \citep[Section 3.3.2 and Section 3.4.4]{absil2009optimization}, the Stiefel manifold ${\rm St}(n,d) := \{\Theta \in \R^{n\times d}: \Theta^\top \Theta = I \}$ and the Grassmann manifold ${\rm Gr}(n,d) = \{ {\rm all~}d{\rm-dimension ~ subspaces ~in~} \R^n\}$ are the embedded manifold and the quotient manifold, respectively. Let $\sim$ denote the equivalence relation on ${\rm St}(n,d)$ defined by 
$$
U \sim V \quad \Leftrightarrow \quad \operatorname{span}(U)=\operatorname{span}(V),
$$
where $\operatorname{span}(X)$ denotes the subspace $\left\{U a: a \in \mathbb{R}^d\right\}$ spanned by the columns of $U \in {\rm St}(n,d)$. The Grassmann manifold ${\rm Gr}(n,d)$ can be regarded as a quotient manifold of ${\rm St}(n,d)$ under the  equivalence relation $\sim$. It is obvious that its total space, i.e., ${\rm St}(n,d)$, is an embedded submanifold of $\R^{n\times d}$. 

\section{Proofs in Subsection \ref{subsec:con}} 
\subsection{Proof of Lemma \ref{lem:upper-bound-grass}} \label{append:lemma}
\begin{proof}[Proof of Lemma \ref{lem:upper-bound-grass}]
Due to the Grassmann constraint, we have $\varphi(V) = \varphi(VQ)$ for any $Q\in \R^{d\times d}$ with $Q^\top Q = QQ^\top = I$. Hence, $\varphi$ is a constant over $S(V):=\{VQ: Q \in \R^{d\times d}, \; Q^\top Q = QQ^\top = I\}$. The tangent space of $S(V)$ at the point $V$ is given by $\{V\Lambda: \Lambda + \Lambda^\top = \revise{\mathbf{0},~ \Lambda\in \mathbb{R}^{d\times d}}\}$. Then, by the definition of the directional derivative, it holds that
\[ \iprod{\nabla \varphi(V)}{V\Lambda} = 0, \]
for all $\Lambda \in \R^{d\times d}$ with $\Lambda + \Lambda^\top = 0$. This implies that \revise{the matrix $V^\top \nabla \varphi(V)$ is symmetric.} 

It follows from the $L_p$-Lipschitz continuity of $\nabla \varphi$ that
\be \label{eq:lsmooth} \varphi(U) \leq \varphi(V) + \iprod{\nabla \varphi(V)}{U-V} + \frac{L}{2}\|U-V\|^2. \ee
Note that $\grad \varphi(V) = (I-VV^\top) \nabla \varphi(V)$. Then,
\be \begin{aligned}
    & \iprod{\grad \varphi(V)}{U - V}  \\
    = & \iprod{\nabla \varphi(V)}{U - V} - \iprod{VV^\top \nabla \varphi(V)}{U-V} \\
    = & \iprod{\nabla \varphi(V)}{U-V} - \iprod{V^\top \nabla \varphi(V)}{V^\top (U-V)} \\
    = & \iprod{\nabla \varphi(V)}{U-V} - \frac{1}{2}\iprod{V^\top \nabla \varphi(V)}{V^\top (U-V) + (U-V)^\top V} \\
    = & \iprod{\nabla \varphi(V)}{U-V} - \frac{1}{2}\iprod{V^\top \nabla \varphi(V)}{(U-V)^\top(U-V)},
\end{aligned}
\ee
where the third equality is due to the symmetry of $V^\top \nabla \varphi(V)$ and the last equality is from $V^\top V = U^\top U = I$. It follows that
\be \label{eq:grad-ineq} \begin{aligned}
    & \iprod{\nabla \varphi(V)}{U-V} \\
    = & \iprod{\grad \varphi(V)}{U - V} + \frac{1}{2}\iprod{V^\top \nabla \varphi(x)}{(U-V)^\top(U-V)} \\
    \leq & \iprod{\grad \varphi(V)}{U - V} + \frac{C_g}{2} \|U-V\|^2.
\end{aligned}
\ee
Combining \eqref{eq:lsmooth} and \eqref{eq:grad-ineq}, we have \eqref{eq:upbound-grass} holds with $\hat{L} = L_g + C_g$. 

For the Riemannian gradient $(I-VV^\top) \nabla \varphi(V)$, it holds that
\[ \begin{aligned}
& \| (I-VV^\top) \nabla \varphi(V) - (I -UU^\top) \nabla \varphi(U) \| \\
= & \| (I -VV^\top) (\nabla \varphi(V) - \nabla \varphi(U)) + (UU^\top - VV^\top) \nabla \varphi(U) \|  \\
\leq & L \| U -V \| + \| U(U-V)^\top\nabla \varphi(U) + (U-V)V^\top \nabla \varphi(V)  \| \\
\leq & (L + 2C_g) \|U-V\|,
\end{aligned}
 \]
 where the first inequality is due to the Lipschitz continuity of $\nabla \varphi(V)$ over ${\rm St}(n,d)$ and the second inequality is from $ \| V \|_2 \leq 1$ for all $V \in {\rm St}(n,d)$. This completes the proof by setting $\tilde{L} = L + 2C_g$. 
\end{proof}

\subsection{Proof of Theorem \ref{thm:convergence}} \label{append:thm}
\begin{proof}[Proof of Theorem \ref{thm:convergence}]
One can prove the conclusion by following the arguments in \citep[Theorem 5.2]{chen2021decentralized}. Let us define
\[ \mathbf{G}_k:= \begin{bmatrix}
g_{1,k} \\
\vdots \\
g_{N,k} \\
\end{bmatrix}, \quad \mathbf{y}_k = \begin{bmatrix}
y_{1,k} \\
\vdots \\
y_{N,k}
\end{bmatrix}, \quad  \hat{y}_k:=\frac{1}{N} \sum_{i=1}^N y_{i,k}, \quad
\mathbf{d}_k = \begin{bmatrix}
d_{1,k} \\
\vdots \\
d_{N,k}
\end{bmatrix},
\quad \hat{d}_k = \frac{1}{N} \sum_{i=1}^N d_{i,k},
\]
% \[ \mathbf{A}_k = \begin{bmatrix}
% \bar{A}_{1,k} \\
% \vdots \\
% \bar{A}_{N,k}
% \end{bmatrix},\quad \hat{A}_k = \frac{1}{N} \sum_{i=1}^N A_{i,k}, \quad \hat{\mathbf{A}}_k = \mathbf{1}_N \otimes \hat{A}_k,\quad \mathbf{B}_k = \begin{bmatrix}
% \bar{B}_{1,k} \\
% \vdots \\
% \bar{B}_{N,k}
% \end{bmatrix}, \quad \hat{B}_k = \frac{1}{N}\sum_{i=1}^N B_{i,k}, 
% \quad \hat{\mathbf{B}}_k = \mathbf{1}_N \otimes \hat{B}_k. 
% \]
\[
\hat{g}_k :=\frac{1}{N} \sum_{i=1}^N g_{i,k}, \quad \hat{\mathbf{G}}_k:= \mathbf{1}_N \otimes \hat{g}_k.  \]
It follows from the construction of $F_{i,k}$ in \eqref{eq:local-FIM} that 
\[ \|d_{i,k}\| \leq \frac{1}{\lambda}\|y_{i,k}\|, \quad \|\mathbf{d}_k\| \leq \frac{1}{\lambda} \|\mathbf{y}_k\|. \]
By the assumption $t \geq \lceil \log_{\sigma_2}\left( \frac{1}{2\sqrt{N}} \right) \rceil$ and following \citep[Lemma 5.1]{chen2021decentralized}, we have 
\be \label{eq:recru-y}\|\mathbf{y}_{k+1} - \hat{\mathbf{G}}_{k+1} \| \leq \sigma_2^t \| \mathbf{y}_k - \hat{\mathbf{G}}_k \| + \| \mathbf{G}_{k+1} - \mathbf{G}_k \| \ee
and 
\be \label{eq:recru-G} \|\mathbf{G}_{k+1} - \mathbf{G}_k\| \leq 2\alpha L_g\|\mathbf{\Theta}_k - \bar{\mathbf{\Theta}}_k\| + \beta / \lambda L_g\|\mathbf{y}_k\|. \ee
In addition, with Assumptions \ref{assum:graph} and \ref{assum:obj}, and small $\beta$ depending on $\delta_1,\delta_2, \sigma_2, \alpha, t, L_g, C_g$, it holds that $\|y_{i,k}\| \leq L_g +2C_g$ for all $i,k$ and $\mathbf{\Theta}_k \in \mathcal{N}$, and there exists a constant $C_1$ depending on $\delta_1,\delta_2, \sigma_2, \alpha, t$ such that 
\[ \frac{1}{N} \| \mathbf{\Theta}_k - \bar{\mathbf{\Theta}}_k \|^2 \leq C_1 (L_g+ 2C_g)^2 \beta^2 / \lambda^2. \]
Hence, $\Theta_k \in \mathcal{N}$ for all $k \geq 0$ if $\beta$ is small enough. From the Lipschitz continuity and the boundedness of ${\rm St}(n,d)$, we have
\[ \| \hat{A}_{i,k} -\hat{A}_{i,k-1}\| \leq 2C_a L_a\| \Theta_{i,k} - \Theta_{i,k-1}\|, \quad  \| \hat{B}_{i,k} -\hat{B}_{i,k-1}\| \leq 2C_b L_b\| \Theta_{i,k} - \Theta_{i,k-1}\|.   \]
Then, using a similar proof to \eqref{eq:recru-y}, one has 
\[ \| \bar{A}_{i,k+1} - \hat{A}_{i,k}  \| \leq \frac{1}{2} \|\bar{A}_{i,k} - \hat{A}_{k-1}\| + 2L_aC_a \| \Theta_{i,k+1} - \Theta_{i,k} \|. \]
By proof of induction and a small $\beta$ depending on $\sigma_2, t, L_a, C_a, \delta_1, \delta_2, \alpha, \lambda$, it holds $ \| \bar{A}_{i,k} \| \leq \kappa_a $ for some $\kappa_a > 0$.
Similarly, we have $\| \bar{B}_{i,k} \| \leq \kappa_b$ for some $\kappa_b > 0$. Hence, by the construction of $F_{i,k}$ in \eqref{eq:local-FIM}, there exists a positive constant $\kappa_F$ such that
\be \label{eq:bounded-F}  \lambda I \preceq F_{i,k} \preceq \kappa_F I.  \ee

Similar to \citep[Equations (E.4) and (E.5)]{chen2021decentralized}, there exists a $\rho_t \in (0,1)$ depending on $\alpha, \sigma_2, t$ such that
\be \label{eq:consensus} 
\|  \mathbf{\Theta}_{k+1} - \bar{\mathbf{\Theta}}_{k+1}  \| \leq \rho_t\| \mathbf{\Theta}_k - \bar{\mathbf{\Theta}}_k  \| + \beta/\lambda \|\mathbf{y}_k\|, \ee
and 
\be \label{eq:triangle} \| \mathbf{y}_k \| \leq \| \mathbf{y}_k -\hat{\mathbf{G}}_k \| + \| \hat{\mathbf{G}}_k \|. \ee

With the inequalities \eqref{eq:recru-y}, \eqref{eq:recru-G}, \eqref{eq:bounded-F}, \eqref{eq:consensus}, and \eqref{eq:triangle}, the key for proving the convergence is to establish the decrease on $\varphi(\bar{\Theta}_k)$. Note that \citep[Lemma E.2]{chen2021decentralized} gives the decrease based on the decentralized Riemannian gradient tracking method. Although we use the natural gradient direction, the decrease can be proved similarly. Here, let us present the sketch of the proofs. Define $H_{i,k} = F_{i,k}^{-1}$ and $\bar{H}_k = \frac{1}{N}\sum_{i=1}^N H_{i,k}$. For the term $\iprod{\hat{g}_k}{\hat{\theta}_{k+1} - \hat{\theta}_k}$, we have
\be \label{eq:b1}
\begin{aligned}
\iprod{\hat{g}_k}{\hat{\theta}_{k+1} - \hat{\theta}_k} = & \iprod{\hat{g}_k}{\hat{\theta}_{k+1} - \hat{\theta}_k - \beta \bar{H}_k \hat{g}_k + \beta \bar{H}_k \hat{g}_k} \\
\leq & -\beta/\kappa_F\| \hat{g}_k\|^2 + \underbrace{\iprod{\hat{g}_k}{\frac{1}{N}\sum_{i=1}^N \left[ \theta_{i,k+1} - (\theta_{i,k} - \beta d_{i,k} -\alpha P_{T_{\theta_{i,k}} \calM}\left(\sum_{j=1}^N w^t_{ij}\theta_{j,k} \right) \right]}}_{s_1} \\
& + \underbrace{\iprod{\hat{g}_k}{\frac{1}{N}\sum_{i=1}^N \left[\beta(\bar{H}_k y_{i,k} - d_{i,k})  - \alpha P_{T_{\theta_{i,k}} \calM}\left(\sum_{j=1}^N w^t_{ij}\theta_{j,k} \right)\right]}}_{s_2}.
\end{aligned}
\ee
By the Lipschitz-like property of $R_{\theta_{i,k}}$ (i.e., there exists a $M > 0$ such that $\| R_{\theta}(u) - \theta + u\|\leq M \|u\|^2$ for any $\theta \in {\rm St}(n,d)$ and $u \in T_{\theta} \calM$ \citep{boumal2019global}), we can bound $s_1$ by $ \mathcal{O}(\|\mathbf{\Theta}_k - \bar{\mathbf{\Theta}}_k\|^2) + \mathcal{O}(\beta^2\|\mathbf{y}_k\|^2)$ from above. For $s_2$, it holds
\[ \begin{aligned}
s_2 & = \iprod{\hat{g}_k}{\frac{1}{N}\sum_{i=1}^N \beta  (\bar{H}_k y_{i,k} - H_{i,k}y_{i,k}) } + \iprod{\hat{g}_k}{\frac{1}{N}\sum_{i=1}^N \beta  ( H_{i,k} y_{i,k} - d_{i,k} )} - \iprod{\hat{g}_k}{ \frac{1}{N}\sum_{i=1}^N \alpha P_{T_{\theta_{i,k}} \calM}\left(\sum_{j=1}^N w^t_{ij}\theta_{j,k} \right)} \\
& = \iprod{\hat{g}_k}{\frac{1}{N}\sum_{i=1}^N \beta  H_{i,k} (\hat{g}_k - y_{i,k}) } + \iprod{\hat{g}_k}{\frac{1}{N}\sum_{i=1}^N \beta P_{N_{\theta_{i,k}} \calM}(H_{i,k} y_{i,k})} - \iprod{\hat{g}_k}{ \frac{1}{N}\sum_{i=1}^N \alpha P_{T_{\theta_{i,k}} \calM}\left(\sum_{j=1}^N w^t_{ij}\theta_{j,k} \right)} \\
& \leq 2 \beta^2 \| \hat{g}_k\|^2 + \mathcal{O}(\| \mathbf{y}_k - \hat{\mathbf{g}}_k  \|^2) + \mathcal{O}(\|\mathbf{\Theta}_k - \bar{\mathbf{\Theta}}_k\|^2)  + \mathcal{O}(\beta^2\|\mathbf{y}_k\|^2),
\end{aligned}  \]
where $N_{\theta}\calM$ denotes the normal space at $\theta$ to $\calM$. Hence, 
\be \iprod{\hat{g}_k}{\hat{\theta}_{k+1} - \hat{\theta}_k} \leq - \beta/\kappa_F \|\hat{g}_k\|^2 + 2\beta^2 \|\hat{g}_k\|^2 + \mathcal{O}(\| \mathbf{y}_k - \hat{\mathbf{g}}_k  \|^2) + \mathcal{O}(\|\mathbf{\Theta}_k - \bar{\mathbf{\Theta}}_k\|^2)  + \mathcal{O}(\beta^2\|\mathbf{y}_k\|^2). \ee
Following the proof in \citep[Lemma E.2]{chen2021decentralized}, one has
\be \label{eq:decrease} \varphi(\bar{\Theta}_{k+1}) \leq  \varphi(\bar{\Theta}_{k}) - \left[ \frac{\beta}{\kappa_F} - (4L_g + 2)\beta^2 \right] \|\hat{g}_k\|^2 + \mathcal{O}(\| \mathbf{y}_k - \hat{\mathbf{g}}_k  \|^2) + \mathcal{O}(\|\mathbf{\Theta}_k - \bar{\mathbf{\Theta}}_k\|^2) + \mathcal{O}(\|\mathbf{\Theta}_{k+1} - \bar{\mathbf{\Theta}}_{k+1}\|^2) + \mathcal{O}(\beta^2\|\mathbf{y}_k\|^2). \ee
Applying \citep[Lemma 2]{xu2015augmented} to \eqref{eq:recru-y} leads to 
\[ \sum_{k=0}^K \| \mathbf{y}_k - \hat{\mathbf{g}}_k \|^2 \leq \mathcal{O}\left(\sum_{k=1}^K\|\mathbf{G}_{k} - \mathbf{G}_{k-1}\|^2\right) + \mathcal{O}(1) \leq \mathcal{O}\left( \sum_{k=0}^K \|\mathbf{\Theta}_k - \bar{\mathbf{\Theta}}_k\|^2 \right) + \mathcal{O}\left(\beta^2\sum_{k=0}^K\|\mathbf{y}_k\|^2\right) + \mathcal{O}(1). \]
Due to \eqref{eq:consensus}, we have
\[ \sum_{k=0}^K \|\mathbf{\Theta}_k - \bar{\mathbf{\Theta}}_k\|^2 \leq   \mathcal{O}\left(\beta^2\sum_{k=0}^K\|\mathbf{y}_k\|^2\right) + \mathcal{O}(1). \]
Then, summing \eqref{eq:decrease} over $k=0,\ldots, K$ gives
\be \label{eq:decrease2} 
 \varphi(\bar{\Theta}_{K+1}) \leq  \varphi(\bar{\Theta}_{k}) - \left[ \frac{\beta}{\kappa_F} - (4L_g + 2)\beta^2 \right] \|\hat{g}_k\|^2 + \mathcal{O}(\beta^2\|\mathbf{y}_k\|^2) + \mathcal{O}(1). \ee
 By a similar technique \citep[Theorem 5.2]{chen2021decentralized} on connecting $\|\mathbf{y}_k\|$ and $\|\hat{g}_k\|$, we conclude 
 $$
\begin{aligned}
& \min _{k \leq K} \frac{1}{N}\left\|\mathbf{\Theta}_k-\bar{\mathbf{\Theta}}_k\right\|^2 \leq \mathcal{O}(\frac{\beta}{K}), \\
& \min_{k \leq K}\left\|\grad \varphi\left(\bar{\Theta}_k\right)\right\|^2 \leq \mathcal{O}(\frac{1}{\beta K}),
\end{aligned}
$$
for sufficient small $\beta$ depending on $\sigma_2, t, L_g, C_g, L_a, L_b, C_a, C_b, \lambda, \delta_1, \delta_2$, and $\alpha$, and the constants in $\mathcal{O}$ are related to $\alpha$, $\sigma_2, t$, $\varphi(\bar{\Theta}_0),$ $\inf_{\Theta \in \calM} \varphi(\Theta)$,  $L_g, C_g, L_a, L_b, C_a, C_b, \lambda, \delta_1$, and $\delta_2$.
 
\end{proof}

\section{Additional comparisons with the centralized Riemannian natural gradient method}\label{sec:app2}

\revise{We test the numerical performance of our method compared with the centralized Riemannian natural gradient method (RNGD) \citep{hu2022riemannian}. Figure \ref{fig:mc:sl2} shows the results on the low-dimension subspace learning problem with Sarcos dataset. We conclude that our algorithm can achieve comparable MSE as RNGD although the convergence speed of our DRNGD is slightly slower than RNGD in terms of \#grad/N. this is due to the inexactness of constructing the Riemannian natural gradient direction. It should be emphasized that the centralized algorithm, RNGD, needs to collect all data into a single agent, which breaks data privacy.}
\begin{figure*}[!htb]
	\begin{center}
		\begin{minipage}[b]{0.4\linewidth}
			\centering
			\centerline{\includegraphics[width=\linewidth]{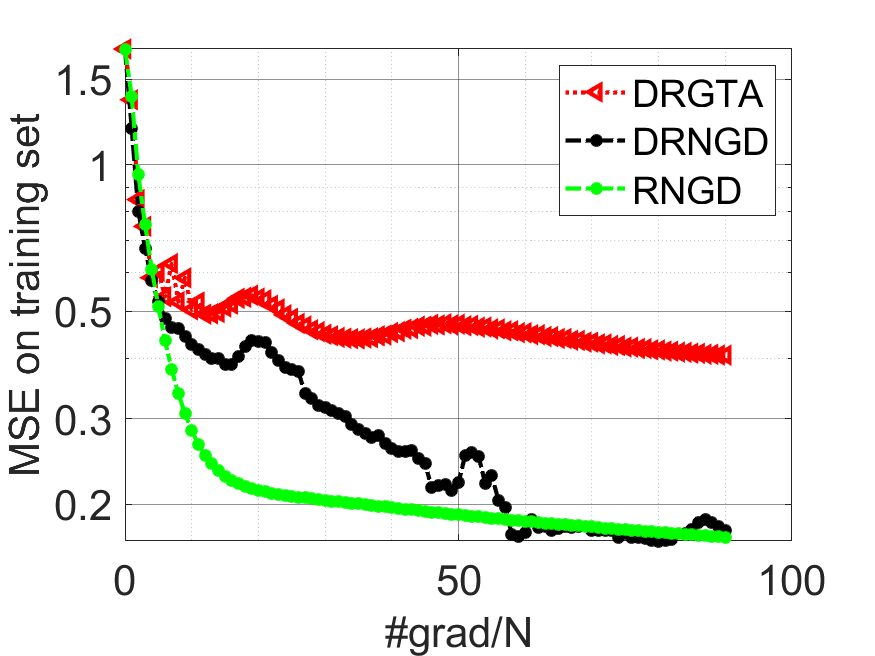}}
			%  \vspace{1.0cm}
			\centerline{}\medskip
		\end{minipage}
		%\hfill
		\begin{minipage}[b]{0.4\linewidth}
			\centering
			\centerline{\includegraphics[width=\linewidth]{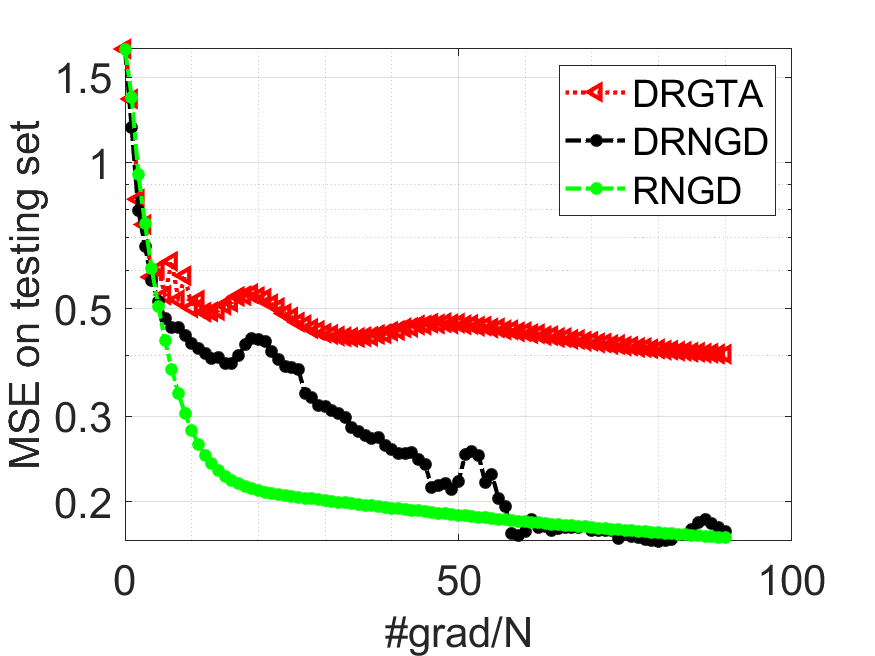}}
			%\cspace{1.5cm}
			\centerline{} \medskip
		\end{minipage}
	\end{center}
	\vskip -0.4in
	\caption{ Numerical results on the low-dimension subspace learning problem with Sarcos dataset.}
	\label{fig:mc:sl2}
\end{figure*}  

\end{appendix}

\end{document}